\documentclass[11pt]{article}

\usepackage[square,numbers]{natbib}
\usepackage[T1]{fontenc}
\usepackage{lmodern}
\usepackage{microtype}

\usepackage{multirow}
\usepackage{amsfonts}
\usepackage{mathrsfs}
\usepackage[title]{appendix}
\usepackage{textcomp}
\usepackage{manyfoot}
\usepackage{booktabs}
\usepackage{algorithmicx}

\usepackage{fullpage}
\usepackage{graphicx}
\usepackage{amssymb}
\usepackage{amsmath}
\usepackage{amsthm}
\usepackage{xcolor}
\usepackage{listings}
\usepackage{color} 
\usepackage{algorithm}
\usepackage{algpseudocode}
\usepackage{thm-restate}
\usepackage{mathtools}

\usepackage{hyperref}
\usepackage[capitalise,nameinlink,noabbrev]{cleveref}

\hypersetup{
    colorlinks,
    citecolor=bbluegray,
    linkcolor=ddarkbrown,
    urlcolor=blue,
    breaklinks
}

\definecolor{ddarkbrown}{rgb}{0.5,0.2,0.05} 
\definecolor{bbluegray}{rgb}{0.05,0,0.5}
\definecolor{mygreen}{RGB}{28,172,0}
\definecolor{mylilas}{RGB}{170,55,241}

\newtheorem{theorem}{Theorem}[section]
\newtheorem{proposition}[theorem]{Proposition}
\newtheorem{definition}[theorem]{Definition}
\newtheorem{lemma}[theorem]{Lemma}

\newtheorem{assumption}[theorem]{Assumption}
\newtheorem{example}[theorem]{Example}
\newtheorem{remark}{Remark}

\newcommand{\innp}[1]{\langle #1 \rangle}
\newcommand{\norm}[1]{\| #1 \|} 
\newcommand{\TxM}{T_x\mathcal{M}}
\newcommand{\TyM}{T_y\mathcal{M}}
\newcommand{\TgtM}{T_{\gamma(t)}\mathcal{M}}
\newcommand\curvtensor{\mathfrak{R}}
\newcommand{\defi}{\stackrel{\mathrm{\scriptscriptstyle def}}{=}}
\newcommand{\abs}[1]{\lvert #1 \rvert} 
\newcommand\M{\mathcal{M}}
\newcommand\expon[1]{\operatorname{Exp}_{#1}}
\newcommand\exponinv[1]{\operatorname{Exp}^{-1}_{#1}} 
\newcommand\ball{\bar{B}}

\newcommand\deltar[1]{\delta_{#1}}
\newcommand\zetar[1]{\zeta_{#1}}
\newcommand*\kmin{\kappa_{\min}}
\newcommand*\kmax{\kappa_{\max}}
\newcommand\B{\mathcal{B}}
\newcommand{\bigo}[1]{O( #1 )}

\newcommand{\bigotilde}[1]{\widetilde{O}( #1 )}

\newcommand{\argmin}{\mathop{\rm argmin}}

\newcommand\blfootnote[1]{%
  \begingroup
  \renewcommand\thefootnote{}\footnote{#1}%
  \addtocounter{footnote}{-1}%
  \endgroup
}

\def\isArxiv{1}

\begin{document}

\title{Strong Convexity of Sets in Riemannian Manifolds}
\date{}

\author{Damien Scieur$^{\ast,\ddagger}$, David Martínez-Rubio$^{\ast,\mathsection}$, Thomas Kerdreux $^{\ast,\mathsection}$,\\ Alexandre d'Aspremont$^{\diamond}$, Sebastian Pokutta$^{\mathsection}$. \blfootnote{Equal Contribution (randomized order).\\
\indent$^\ddagger$Samsung SAIL Montreal, Canada.\\
\indent$^\diamond$CNRS \& D.I. \'Ecole Normale Sup\'erieure, Paris, France.\\
\indent$^\mathsection$Zuse Institute Berlin \& Technische Universit{\"a}t, Berlin, Germany.}}

\maketitle

    \abstract{Curvature properties of convex objects, such as strong convexity, are important in designing and analyzing convex optimization algorithms in the Hilbertian or Riemannian settings. In the case of the Hilbertian setting, strongly convex sets are well studied. Herein, we propose various definitions of strong convexity for uniquely geodesic sets in a Riemannian manifold. We study their relationship, propose tools to determine the geodesic strongly convex nature of sets, and analyze the convergence of optimization algorithms over those sets. In particular, we demonstrate that the Riemannian Frank-Wolfe algorithm enjoys a global linear convergence rate when the Riemannian scaling inequalities hold.}

{
    \setcounter{tocdepth}{1}
    \small
    \tableofcontents
}

\section{Introduction}
\textit{Strong convexity} is a fundamental property characterizing convex objects such as normed spaces, functions, or sets. Algorithms can leverage strong (or uniform) convexity structures in optimization and learning problems to accelerate convergence, enhance generalization bounds, and strengthen concentration properties.

For instance, uniform convexity (that interpolates between plain and strong convexity) has a direct effect on the concentration inequalities of vector-valued martingales when measured via this space norm \citep{pisier1975martingales,pisier2011martingales}.
Also, assuming strong convexity for objective functions has become standard for studying optimization algorithm performance in Hilbert spaces and proving generalization bounds.
Similarly, the strong convexity of sets plays a role in optimization algorithms \citep{demyanov1970,garber2015faster,Molinaro20,kerdreux2021projection,kerdreux2021affine} or learning algorithms, e.g., to characterize action sets in online learning \citep{huang2017following,Molinaro20,kerdreux2021local,dekel2017online,bhaskara2020online,bhaskara2020onlineMany} or bandit algorithms \citep{bubeck2018sparsity,kerdreux2021linear}.
These notions have a well-understood and well-exploited generalization in normed space settings  \citep{goncharov2017strong}.

However, many machine-learning problems are being formulated and solved beyond normed spaces, e.g., metric spaces and Riemannian manifolds.
For instance, Sinkhorn divergences \citep{cuturi2013sinkhorn,feydy2019interpolating} are regularized formulations of optimal transport distances \citep{villani2009optimal}. Those methods approximate the optimal transport distance, breaking some computational bottlenecks.
This has resulted in modeling many problems using these geometries over probability spaces, and hence, metric spaces.
Similarly, some progress has recently been made in optimization algorithms on Riemannian manifolds \citep{sun2016complete,allen2018operator,hosseini2019alternative,hosseini2020recent}.
The definition of these optimization problems on these structures produces local-global properties when identifying geodesically convex problems that may be non-convex.

Unfortunately, the notion of strong convexity in these more complex frameworks is more difficult to formalize.
Several definitions of strong convexity for metric spaces have been developed, such as those for the most notable \textit{non-positively curved}, also known as $\text{CAT}(0)$ spaces \citep{Alek57}. 
Those spaces require the square of the distance function to be geodesically strongly convex.
Similar results, such as those in the Hilbertian case, have been obtained with concentration results \citep{gouic2019fast,ahidar2020convergence,merigot2020quantitative} or in online learning \citep{Paris20,Paris21}.
Nevertheless, the situation is not as well-behaved as in normed spaces, as several non-equivalent notions emerge when describing the problem structures.

\textbf{Motivation.} There are multiple consequences of analyzing strongly convex sets defined in manifolds. In the Euclidean case, such an analysis made the identification of the strong convexity of sets easier: sublevel sets of smooth, strong convex functions are strongly convex sets. In addition, the Frank--Wolfe (FW) algorithm \cite{braun2022conditional} converges faster over strongly convex sets. Generalizing the notion of strong convexity for functions in metric spaces is relatively straightforward (\cref{def:strongly_cvx_geo_fct}). 
However, to our knowledge, this notion has not been studied in the case of sets. Therefore, we propose to fill this gap in the context of a Riemannian manifold.

\textbf{Contributions.} We introduce novel structural properties of sets defined in Riemannian manifolds, provide examples, and demonstrate improved algorithmic performance when minimizing functions over those sets.

\begin{enumerate}
    \item In \cref{sec:geodesic_strong_convexity}, we introduce different definitions of strong convexity of a set in Riemannian manifolds. 
    Each definition relies on two perspectives of uniquely geodesic subsets, either seen as a geodesic metric space or a smooth manifold. 
    \item We establish some relationships between the aforementioned various definitions of strongly convex sets in \cref{sec:relation_between_definitions} and introduce the notion of approximate scaling inequality and its link with what we call \textit{double geodesic strong convexity} property (\cref{def:double_geodesic_strong_convexity}) in \cref{sec:double_geodesic_convexity}.
    \item In \cref{sec:level_set_geodesically}, we prove that sublevel sets of geodesically smooth, strongly convex functions are strongly convex in Riemannian manifolds under mild conditions.
    Similar results have been obtained in the case of Hilbert spaces \citep{journee2010generalized,garber2015faster}.
    This perspective is valuable for identifying strongly convex sets in Riemannian manifolds.
    \item We also provide examples for the strongest of our notions of Riemannian strong geodesic convexity. Hence, those examples also satisfy all other other notions, cf. \cref{ex:riemannian_strongly_convex_sets}.
    \item In \cref{sec:FW_geodesically_strongly_convex_set}, we derive a global linear convergence bound for the Riemannian FW algorithm when the constraint set satisfies the \textit{Riemannian scaling inequality} (\cref{def:riemannian_scaling_inequality}).
\end{enumerate}

\section{Preliminaries and Notations}\label{sec:preliminaries}
Our primary goal is defining several notions of strong convexity for sets in metric spaces that are well-suited for analyzing and designing optimization algorithms.
In this study, we first focus on the case of complete and connected $n$-dimensional Riemannian manifolds.
This setting combines a simple manifold structure, e.g., $n$-dimensional tangent spaces with scalar products, with a \textit{geodesically metric space} structure. 
We now recall some useful concepts.

\subsection{Geodesic Metric Spaces}
\paragraph*{Metric space} A \textit{metric space} is a pair $(\mathcal{M}, d_{\mathcal{M}})$, where $\mathcal{M}$ is a set, and $d_{\mathcal{M}}:\mathcal{M}\times\mathcal{M}\rightarrow \mathbb{R}^+$ is a distance function (the metric), i.e., it satisfies positivity, symmetry, and the triangle inequality.

\paragraph*{Geodesics} A geodesic $\gamma(t):[0,1]\rightarrow \mathcal{M}$ between two points $x,\,y\in\mathcal{M}$ is a smooth curve such that $\gamma(0)=x$, $\gamma(1)=y$, and $\gamma''(t)=0$.
A metric set is a \textit{geodesic metric space} if, for all $x,y\in\mathcal{M}$, there exists a geodesic that connects $x$ and $y$. When this geodesic is unique, the set is considered uniquely geodesic.

\paragraph*{Cartan--Hadamard spaces} In particular, \textit{Cartan--Hadamard spaces} are geodesic metric spaces generalizing Hilbert spaces which are well-suited for convex optimization purposes \citep{jost2012nonpositive,bridson2013metric,bacak2014convex,bacak2018old}.
By definition, a \mbox{Cartan--Hadamard} space $\mathcal{M}$ (a.k.a. \textit{CAT(0) space} or \textit{space of non-positive curvature}) is a geodesically metric space that is \textit{non-positively curved}. Non-positively curved spaces have the particularity that, for all geodesics $\gamma$ and for all elements $x\in\mathcal{M}$, the metric $d_{\mathcal{M}}$ satisfies (e.g., \citep[Corollary 2.5.]{Paris21})
\begin{equation}\label{eq:non_positive_curvature}\tag{non-positive curvature}
    d_{\mathcal{M}}(x,\gamma(t))^2\leq (1-t) d_{\mathcal{M}}(x,\gamma(0))^2 + t d_{\mathcal{M}}(x, \gamma(1))^2 - t (1-t) d_{\mathcal{M}}(\gamma(0),\gamma(1))^2.
\end{equation}
This non-positive curvature indicates that the function $f(\cdot)=d_{\mathcal{M}}(\cdot, x)^2$ is geodesically strongly convex (see \cref{def:strongly_cvx_geo_fct}) for all reference point $x\in\mathcal{M}$.

\subsection{Riemannian Manifolds}

In the case of Banach spaces, wherein the strong convexity has been crucial to algorithm analysis, significantly fewer tools have been developed for directly characterizing sets that appear in geodesic metric spaces and, therefore, in Riemannian manifolds. Those tools are needed to analyze online learning algorithms or some constrained optimization algorithms. We recall some key concepts below.

Throughout this paper, we consider an $n$-dimensional Riemannian manifold $\mathcal{M}$, such as the Stiefel manifold $\operatorname{St}(n,p)$ for $p < n$; the group of rotations $\operatorname{SO}(n)$; hyperbolic spaces; spheres; or the manifold of symmetric positive definite matrices. 

\paragraph*{Connected manifolds} A manifold $\mathcal{M}$ is \textit{connected} if, for all $x,y \in \mathcal{M}$, there exists a continuous path joining those points. An $n$-dimensional manifold is a topological space that is locally homeomorphic to the $n$-dimensional Euclidean space.

\paragraph*{Tangent space} For a point $x\in\mathcal{M}$, a tangent vector is the tangent of a parameterized curve passing through $x$, and all the tangent vectors at $x$ form the tangent space $\TxM $.

\paragraph*{Metric tensor and inner product} A Riemannian manifold $(\mathcal{M}, g_x)$ is such that the Riemannian metric $g$ is a $C^{\infty}$ metric tensor, where for all point $x\in\mathcal{M}$, $g_x$ defines a positive definite inner product in the tangent space $\TxM $, and $x\mapsto g_x$ is $C^{\infty}$.
For $u,v\in\mathcal{M}$, we write $g_x(u,v):=\langle u;v\rangle_x$, and $\|u\|_x : = \langle u ; u \rangle_x$. 

\paragraph*{Complete Riemannian manifold and Riemannian metric} A connected Riemannian manifold is also a metric space.
Furthermore, the length of a continuous piecewise smooth path $\gamma: [0,1]\rightarrow \mathcal{M}$ is defined as $L(\gamma):=\int_{0}^{1}\|\gamma^\prime(t)\|_{\gamma(t)}dt$.
We denote $\mathcal{P}_{x,y}$ as the set of continuous piecewise smooth paths joining $x$ and $y$. We have the \textit{Riemannian metric} $d_{\mathcal{M}}(x,y) := \inf_{\gamma\in\mathcal{P}_{x,y}}L(\gamma)$ as a distance \citep[Theorem 10.2]{boumal2020introduction} and ($\mathcal{M},d_{\mathcal{M}})$ is a metric space.
We refer to \textit{complete manifold} as a manifold that is complete as a metric space.
Moreover, \citep[Theorem 10.8]{boumal2020introduction} states that any connected and complete Riemannian manifold is a geodesic metric space.
In particular, the infimum in the definition of the Riemannian metric is attained at a geodesic (w.r.t. $d_{\mathcal{M}}$) called the \textit{minimizing geodesic}.

\subsection{Sets in Riemannian Manifolds}

\textbf{Importance of sets when optimizing on Manifolds.} When $\mathcal{M}$ is a connected compact Riemannian manifold, and $f$ is a geodesically convex function on $\mathcal{M}$ (\cref{def:strongly_cvx_geo_fct}), then $f$ is constant \citep[Corollary 11.10]{boumal2020introduction}. 
This result motivates the optimization of geodesically convex functions over subsets $\mathcal{C}$ of these manifolds. 

\paragraph*{Assumptions for Sets in Riemannian Manifolds}

In this paper, we will make the assumption that we are working on a subset $\mathcal{C}\subset \mathcal{M}$ that is \emph{uniquely geodesically convex}, defined as the set $\mathcal{C}$ being both convex and uniquely geodesic.

\begin{definition}[Geodesic Convex Closed Subset of a Manifold] 
We say that a closed subset $\mathcal{C}$ of a manifold $\mathcal{M}$ is convex if for any two points $x, y \in \mathcal{C}$, there exists a geodesic from $x$ to $y$ that is distance minimizing and is contained in $\mathcal{C}$.
\end{definition}

\begin{definition}[Uniquely Geodesic Subset of a Manifold] 
We say that a subset $\mathcal{C}$ of a manifold $\mathcal{M}$ is uniquely geodesic if for every two points $x, y \in \mathcal{C}$, there is only one geodesic between $x$ and $y$ that is contained in $\mathcal{C}$.
\end{definition}

\begin{assumption}
    The set $\mathcal{C}\subseteq\mathcal{M}$ is compact, convex, and uniquely geodesic. 
\end{assumption}

For instance, open hemispheres and Cartan--Hadamard manifolds, that is, complete simply connected manifolds of non-positive sectional curvature everywhere, are uniquely geodesic. Hence, their compact counterparts are geodesically convex subsets. However, a closed hemisphere is not uniquely geodesic.

The \textit{exponential map} parameterizes the manifold by mapping vectors in the tangent spaces to $\mathcal{M}$.
For $x\in\mathcal{M}$, it is in general a local diffeomorphism; however, when $\mathcal{M}$ is complete \citep{hopf1931begriff}, it is defined on the whole tangent space as follows.

\begin{definition}[Exponential Map]
Let $\mathcal{M}$ be a Riemannian manifold. For all $x\in\mathcal{M}$ and $v\in \TxM $, let the geodesic $\gamma:[0,1]\rightarrow \mathcal{M}$ satisfying $\gamma(0)=x$ with $\gamma'(0)=v$. The exponential map at x, $\operatorname{Exp}_x:\TxM \rightarrow\mathcal{M}$, is defined as $\operatorname{Exp}_x(v): =\gamma(1)$.
\end{definition}

\begin{remark}[Bijective Exponential Map and Logarithmic Map in the Right Domain]\label{remark:inverse_exponential_map}
    For uniquely geodesically convex sets $\mathcal{C}$ and $x \in \mathcal{C}$, define the set 
\[
\mathcal{C}_x\defi \{v \in T_x\mathcal{M}: t \mapsto \operatorname{Exp}_x(tv), \text{ for } t \in [0, 1], \text{ is minimizing and } \operatorname{Exp}_x(v) \in \mathcal{C} \},
\] 
    then, the exponential map from $x$ restricted to this set is well defined and bijective in $C$, that is, $\operatorname{Exp}_x(y) : \mathcal{C}_x \rightarrow \mathcal{C}$ is a bijective function. In this work we will always refer to the inverse of the exponential map, or logarithmic map, as the inverse of this restriction: $\operatorname{Exp}_x^{-1}: \mathcal{C} \to \mathcal{C}_x$. Note that generally in the literature the logarithmic map is not defined using this restriction.
\end{remark} 

The exponential map links the manifold and its tangent spaces, but vectors belonging to different tangent spaces are not directly comparable. Using the Levi-Civita connection, we can define a parallel transport operator between two points $x,y\in\M$ in the manifold \citep{ambrose1956parallel}.

\begin{definition}[Parallel Transport Operator]
For all $x,y\in\M$, the parallel transport map $\Gamma_x^y:\TxM \rightarrow \TyM $ combines vectors from different tangent spaces by transporting them along geodesics and such that $\langle \Gamma_x^y u; \Gamma_x^y v \rangle_y = \langle u; v\rangle_x$ for all $(u,v)\in \TxM $.
\end{definition}

\subsection{Functions over Manifolds} 

We first define the Riemannian gradient of $f$ at $x$ as the unique element $\nabla f(x) \in \TxM $, such that the directional derivative $Df(x)[v] =  \langle v; \nabla f(x) \rangle_x$ for all  $v\in \TxM $.

We now recall the notion of relative geodesic (strong) convexity and smoothness for a function \mbox{$f:\mathcal{M}\rightarrow\mathbb{R}$} w.r.t. a distance function $d:\mathcal{M}\times \mathcal{M} \rightarrow \mathbb R^+$ (see, for instance, \citep{zhang2016first}).

\begin{definition}[Distance function] A function $d:\mathcal{M}\times\mathcal{M}\rightarrow\mathbb{R}^+$ is called a distance function if it satisfies positivity, symmetry, and the triangle inequality.
\end{definition}

\begin{remark}
The distance $d$ could differ from the distance function $d_{\mathcal{M}}$ associated with the manifold $\mathcal{M}$.
\end{remark}

\begin{definition}[Geodesic (Strong) Convexity]\label{def:strongly_cvx_geo_fct}
    Let $\mathcal{C}$ be a uniquely geodesic set. A function $f:\mathcal{C}\subseteq\mathcal{M}\rightarrow\mathbb{R}$ \emph{is geodesically $\mu$-strongly convex} in $\mathcal{C}$ (resp. convex if $\mu=0)$ w.r.t. $d$ if, for all $x,y\in\mathcal{C}$, we have
\begin{equation}\label{eq:strongly_cvx_geo_fct}
\forall t\in[0,1],\;\;f(\gamma(t)) \leq (1-t)f(x) + t f(y) - \tfrac{\mu}{2} t (1-t) d^2(x,y),\;\;\mu\geq 0.
\end{equation}
\end{definition}

\begin{definition}[Geodesic Smoothness]\label{def:geodesic_smoothness_fct}
    Let $\mathcal{C}$ be a uniquely geodesic set. A function $f:\mathcal{C}\subseteq\mathcal{M}\rightarrow\mathbb{R}$ is \emph{geodesically $L$-smooth} in $\mathcal{C}$ w.r.t. $d$ if, for all $x, y\in\mathcal{C}$, we have
\[
    \abs{f(y) -  f(x) - \langle \nabla f(x); \operatorname{Exp}^{-1}_x(y) \rangle_x} \leq  \tfrac{L}{2}d^2(x,y),\;\; L\geq 0.
\]
\end{definition}
Note that when the function is also geodesic (strongly) convex, the smoothness condition can be simplified by removing the absolute value. Indeed, (strong) convexity implies that the argument of the absolute value is nonnegative.

Finally, we state the following assumption on the relation between the distance function used to define smoothness and strong convexity and the distance function associated with the manifold $\mathcal{M}$. This assumption allows us to use distances whose value is within some constant factors from the Riemannian distance $\norm{\exponinv{x}(y)}$, which corresponds to the definition below with $\ell_\M = L_\M = 1$.

\begin{assumption}[Distances Equivalence]\label{ass:equivalence_distance_exponential}
Let $\mathcal{M}$ be a Riemannian manifold and let $d:\M\times\M\to\mathbb{R}$ be a distance function, possibly different from the Riemannian distance.
There exists $0< l_{\mathcal{M}}\leq L_\mathcal{M}$ such that for all pair $(x,y)\in\mathcal{M}^2$,
\begin{equation}\label{eq:regularity_exponential_map}
\ell_{\mathcal{M}} \|\operatorname{Exp}^{-1}_x(y)\|_x \leq d(x,y) \leq L_{\mathcal{M}}\|\operatorname{Exp}^{-1}_x(y)\|_x.
\end{equation}
\end{assumption}

\section{Strong Convexity of Sets in Riemannian Manifolds}\label{sec:geodesic_strong_convexity}

This section provides several variants of the definition of a strongly convex set in Riemannian manifolds. Some variants are instrumental in analyzing algorithms or proving a set's strongly convex nature.

\subsection{Strong Convexity in Hilbert Spaces}
Many equivalent characterizations of a set's strong convexity exist in Hilbert spaces \citep{goncharov2017strong}. We recall two characterizations in \cref{thm:equivalence_euclidean}. Intuitively, a set is called strongly convex if, for all straight lines in the set and all points $p$ in the line, there exists a ball around $p$ of a certain radius that is contained in the set.

\begin{restatable}[Equivalent Notions of Strong Convexity of Sets in Hilbert Spaces \citep{goncharov2017strong}]{proposition}{equivalenceeuclidean}\label{thm:equivalence_euclidean}
    Consider a compact convex set $\mathcal{C} \subset \mathbb{R}^n$, $\alpha>0$, and a norm $\|x\| = \sqrt{\innp{x, x}}$. 

We say that the set $\mathcal{C}$ is $\alpha$-strongly convex w.r.t. $\|\cdot\|$ if and only if it satisfies the following equivalent assertions.\\
    (a) For all $(x,y)\in\mathcal{C}$, $t\in [0,1]$, and $z\in \mathbb{R}^n$ s.t. $\|z\|=1$, we have
    \begin{equation}\label{eq:strong_convexity_euclidean}
        (1-t) x + t y + \alpha (1-t) t \|x-y\|^2 z\in\mathcal{C}.
    \end{equation}
    (b) For all $(x,v)\in\mathcal{C}\times\partial\mathcal{C}$ and $w\in N_{\mathcal{C}}(v)$ (the normal cone of $\mathcal{C}$ at $x$), we have
    \begin{equation}\label{eq:scaling_inequality}\tag{scaling inequality}
        \langle w; v - x \rangle \geq \alpha \|w\| \|v - x\|^2.
    \end{equation}
\end{restatable}

Going back to the original definitions in the literature, one can find these properties where $\alpha$ takes the value $1 /2R$, where $R > 0$ is the radius of some balls that are used for an alternative equivalent definition of strong convexity of sets \citep{vial1982strong,vial1983strong,goncharov2017strong}. We use $\alpha$ in this work since it is simpler for our purposes.
We also note that \eqref{eq:scaling_inequality} is equivalent to the following condition \citep[Lemma 2.1]{kerdreux2021projection}:
    \begin{equation}\label{eq:scaling_inequality2}
        \langle w; v - x \rangle \geq \alpha \|w\| \|v - x\|^2 \;\; \text{where}\;\; v \in \underset{z\in\mathcal{C}}{\operatorname{argmax}} \langle w; z\rangle.
    \end{equation}

\subsection{Strong Convexity of Sets in Riemannian Manifolds}

In this section, we propose several definitions that extend \eqref{eq:strong_convexity_euclidean} for Riemannian manifolds, which all collapse to the known strong convexity notion when $\mathcal{M}$ is the Euclidean space. We study their relationships and provide examples in \cref{ex:riemannian_strongly_convex_sets} for the strongest conditions; therefore, the examples apply to all of our definitions. 

\subsubsection{Geodesic Strong Convexity} 

\cref{def:geodesic_strong_convexity} adapts \eqref{eq:strong_convexity_euclidean} but relies on the geodesic metric structure of $\mathcal{M}$ only, without considering the Riemannian metric structure (i.e., the definition does not even use the tangent space). We refer to it as the \textit{geodesic strong convexity}.

\begin{definition}[Geodesic Strong Convexity]\label{def:geodesic_strong_convexity}
Let $\mathcal{M}$ be a Riemannian manifold, and let  $\mathcal{C}\subseteq\mathcal{M}$ be a uniquely geodesic set.
    The set $\mathcal{C}$ is geodesically $\alpha$-strongly convex w.r.t. the distance function $d:\mathcal{M}\times\mathcal{M}\rightarrow\mathbb{R}^+$ if, for every geodesic $\gamma$ joining $x,y\in\mathcal{C}$ and every $t\in[0,1]$, we have that the following ball is in $\mathcal{C}$:
\begin{equation}\label{eq:distance_geodesic_strong_convexity}
    \{z\in\mathcal{M} \mid d(\gamma(t), z) \leq \alpha t(1-t)d^2(x,y)\}\subseteq \mathcal{C}.
\end{equation}
\end{definition}

\subsubsection{Riemannian Strong Convexity} The following definition leverages the Riemannian structure of $\mathcal{M}$ via an assumption on the exponential map. The definition states that the inverse image of the set $\mathcal{C}\subset\mathcal{M}$ by the inverse exponential map at each $x\in\mathcal{M}$ must be strongly convex in $\TxM $ for all $x$ in the Euclidean sense. Recall that for a uniquely geodesic set $\mathcal{C}$, the inverse exponential map is always well defined for any two points in $\mathcal{C}$, cf. \cref{remark:inverse_exponential_map}.

\begin{definition}[Riemannian Strong Convexity]\label{def:riemannian_strong_convexity}
Let $\mathcal{M}$ be a Riemannian manifold, and $\mathcal{C}\subseteq\mathcal{M}$ be a uniquely geodesic set.
Then, $\mathcal{C}$ is a Riemannian $\alpha$-strongly convex set if, for all $x\in\mathcal{C}$, the set
\[
    \operatorname{Exp}_x^{-1}(\mathcal{C}) := \{ y \in \TxM : y = \operatorname{Exp}^{-1}_x(z),\;\;z\in\mathcal{C}  \}
\]
is $\alpha$-strongly convex w.r.t. $\|\cdot\|_x$ in the Euclidean sense \eqref{eq:strong_convexity_euclidean}.
\end{definition}

\subsubsection{Double Geodesic Strong Convexity} In \cref{def:double_geodesic_strong_convexity}, we now leverage the Riemannian structure through the exponential map to provide another notion of strong convexity of a set, analogous to the Euclidean formulation in \eqref{eq:strong_convexity_euclidean}.
For $t\in[0,1]$, we make the parallel between the term $t x + (1-t)y$ in \eqref{eq:strong_convexity_euclidean} and $\gamma(t)$, the geodesic $\gamma$ joining $y$, and $x$ in $\mathcal{M}$. 
Then, \eqref{eq:strong_convexity_euclidean} in $\mathcal{M}$ ensures that, for all $z\in \TgtM $ with unit norm $\norm{z}=1$, we have $\operatorname{Exp}_{\gamma(t)}\big(\alpha t(1-t)d^2(x,y)z\big) \in \mathcal{C}$.

\begin{definition}[Double Geodesic Strong Convexity]\label{def:double_geodesic_strong_convexity}
Let $\mathcal{M}$ be a Riemannian manifold equipped with a distance function $d(\cdot,\cdot)$, and let $\mathcal{C}\subseteq\mathcal{M}$ be a set that is uniquely geodesic.
    The set $\mathcal{C}$ is a double geodesically $\alpha$-strongly convex set w.r.t. $d(\cdot, \cdot)$ if, for every geodesic $\gamma$ joining $x,y\in\mathcal{C}$,
\begin{equation}\label{eq:strong_convexity_riemannian}
    \forall t\in[0,1],\;\;\forall z\in \TgtM :\|z\|_{\gamma(t)} \leq  \alpha t (1-t) d^2(x,y) \quad \Rightarrow \quad \operatorname{Exp}_{\gamma(t)}(z) \text{ exists and is in }\mathcal{C}.
\end{equation}
\end{definition}

The double geodesic strong convexity can also be rewritten in terms of the exponential map if $\operatorname{Exp}_{\gamma(t)}(\alpha t (1-t) d^2(x,y)z)\in\mathcal{C}$ for all $z$ of unit norm in $\TgtM $.
In this manner, it mirrors the algebraic expression of the Euclidean definition we provided in \eqref{eq:strong_convexity_euclidean} but within the Riemannian setup.
We denote it as the \textit{double geodesic strong convexity} also because the point $\operatorname{Exp}_{\gamma(t)}(\alpha t (1-t) d^2(x,y)z)$ is built via two geodesics, one between $x$ and $y$, and another starting at $\gamma(t)$.
In \cref{sec:double_geodesic_convexity}, we define a characterization of the double geodesic convexity via the classical \textit{double exponential map} (\cref{def:double_exponential_map}).

\subsubsection{Riemannian Scaling Inequality} 
In Euclidean space, \eqref{eq:scaling_inequality} is an equivalent definition  of the strong convexity of the sets via \cref{thm:equivalence_euclidean}, which helps in establishing convergence proofs of various algorithms.
In \cref{def:riemannian_scaling_inequality}, we propose the notion of the Riemannian scaling inequality, which is the the Riemannian counterpart of~\eqref{eq:scaling_inequality2}.

\begin{definition}[Riemannian Scaling Inequality]\label{def:riemannian_scaling_inequality}
Let $\mathcal{M}$ be a Riemannian manifold, and let $\mathcal{C}\subset\mathcal{M}$ be compact and uniquely geodesically convex.
The elements in the set $\mathcal{C}$ then satisfy the Riemannian scaling inequality if, for some $\alpha > 0$, for all $x\in\mathcal{C}$, $w\in \TxM $, and $v \in \underset{z\in\mathcal{C}}{\operatorname{argmax}} \langle w; \operatorname{Exp}^{-1}_x(z)\rangle_x$,
\begin{equation}\label{eq:scaling_inequality_riemannian}\tag{Riemannian scaling inequality}
    \langle w; \operatorname{Exp}^{-1}_x(v)\rangle_x \geq \alpha \|w\|_x \|\operatorname{Exp}_x^{-1}(v)\|_x^2.
\end{equation}
\end{definition}

\section{Relations between the Definitions of Strong Convexity}\label{sec:relation_between_definitions}

This section establishes some implications and equivalences between strong convexity notions and scaling inequality for manifolds. We summarize the links between these notions here.
\begin{align}\nonumber
    & \begin{matrix}
        \text{\scriptsize \hyperref[def:riemannian_strong_convexity]{\normalcolor Riemannian}}\\
        \text{\scriptsize \hyperref[def:riemannian_strong_convexity]{\normalcolor Strong Convexity}}
    \end{matrix}  
    \xRightarrow{\begin{matrix}
    \text{\footnotesize  Prop. \ref{prop:riemannian_implies_geodesic},}\vspace{-1ex}\\\text{\footnotesize  Hadamard}
    \end{matrix}} 
    \begin{matrix}
        \text{\scriptsize \hyperref[def:geodesic_strong_convexity]{\normalcolor Geodesic Strong}}\\
        \text{\scriptsize \hyperref[def:geodesic_strong_convexity]{\normalcolor Convexity}}
    \end{matrix} 
    \xLeftrightarrow{\begin{matrix}
    \text{\footnotesize  Prop. \ref{prop:equivalence_double_str_cvx_hadamard}}\vspace{-1ex} \\ \text{\footnotesize A. \ref{ass:equivalence_distance_exponential}}\end{matrix}}  
    \begin{matrix}
        \text{\scriptsize \hyperref[def:double_geodesic_strong_convexity]{\normalcolor Double Geodesic} }\\ \text{\scriptsize \hyperref[def:double_geodesic_strong_convexity]{\normalcolor Strong Convexity}  } 
    \end{matrix}
    \xRightarrow{\text{\footnotesize Prop. \ref{prop:implication_approximate_riemannian_scaling_inequality}}}    
     \begin{matrix}
        \text{\scriptsize \hyperref[def:approximate_riemannian_scaling_inequality]{\normalcolor Approx. Riemannian} }\\\text{\scriptsize \hyperref[def:approximate_riemannian_scaling_inequality]{\normalcolor Scaling Inequality} }
    \end{matrix}\\
    & \begin{matrix}
        \text{\scriptsize \hyperref[def:riemannian_strong_convexity]{\normalcolor Riemannian}}\\
        \text{\scriptsize \hyperref[def:riemannian_strong_convexity]{\normalcolor Strong Convexity}}
    \end{matrix}
    \xRightarrow{\text{\footnotesize Prop. \ref{prop:riemannian_scaling_inequality}}}
    \begin{matrix}
        \text{\scriptsize \hyperref[def:riemannian_scaling_inequality]{\normalcolor Riemannian}}\\\text{\scriptsize \hyperref[def:riemannian_scaling_inequality]{\normalcolor Scaling Inequality} }
    \end{matrix}  
    \xRightarrow{\hspace{5ex}}
    \begin{matrix}
        \text{\scriptsize \hyperref[def:approximate_riemannian_scaling_inequality]{\normalcolor Approx. Riemannian} }\\\text{\scriptsize \hyperref[def:approximate_riemannian_scaling_inequality]{\normalcolor Scaling Inequality} }
    \end{matrix} \label{eq:implication_def}
\end{align}

The scaling inequality and global strong convexity of a set are equivalent notions (\cref{thm:equivalence_euclidean}) in Hilbert spaces, but this won't be the case when working with Riemannian manifolds.
Instead, \cref{prop:riemannian_scaling_inequality} states that the Riemannian strong convexity implies a Riemannian scaling inequality, the latter being valuable for analyzing algorithms (\cref{sec:FW_geodesically_strongly_convex_set}).

\begin{restatable}[Riemannian Strong Convexity implies Riemannian Scaling Inequality]{proposition}{riemannianscalinginequality}\label{prop:riemannian_scaling_inequality} Let $\mathcal{M}$ be a Riemannian manifold, and $\mathcal{C}\subset\mathcal{M}$ be compact and uniquely geodesically convex. 
Let us assume that $\mathcal{C}$ is a Riemannian $\alpha$-strongly convex set (\cref{def:riemannian_strong_convexity}).
The elements in the set $\mathcal{C}$ then also satisfy the Riemannian scaling inequality (\cref{def:riemannian_scaling_inequality}).
\end{restatable}

\begin{proof}
    As the set is strongly convex in the Euclidean sense, we obtain from \eqref{eq:scaling_inequality} that, for all $x\in\mathcal{C}$ and for all $w \in \TxM $, we have
    \[
        \langle w; u \rangle_x \geq \alpha \|w\|_x \|u - x\|_x^2,\quad u \in \underset{z\in\operatorname{Exp}_x^{-1}(\mathcal{C})}{\operatorname{argmax}} \langle w; z\rangle_x.
    \]
    As the exponential map is bijective over the set $\mathcal{C}$, 
    \[
        \underset{z\in\operatorname{Exp}_x^{-1}(\mathcal{C})}{\operatorname{argmax}} \langle w; z\rangle_x = \operatorname{Exp}_x^{-1}\left( \underset{z\in\mathcal{C}}{\operatorname{argmax}} \langle w; \operatorname{Exp}_x^{-1}(z)\rangle_x\right).
    \]
   Therefore, using $v = \operatorname{Exp}_x(u)$ gives us
    \[
        \langle w; \operatorname{Exp}^{-1}_x(v) \rangle_x \geq \alpha \|w\|_x \|\operatorname{Exp}^{-1}_x(v)\|_x^2,\quad v \in \underset{z\in\mathcal{C}}{\operatorname{argmax}} \langle w; \operatorname{Exp}_x^{-1}(z)\rangle_x.
    \]
\end{proof}

\begin{restatable}[Riemannian Strong Convexity implies Geodesic Strong Convexity]{proposition}{riemannianimpliesgeodesic}\label{prop:riemannian_implies_geodesic}
    Let $\mathcal{M}$ be a Riemannian manifold, and let $\mathcal{C}\subset\mathcal{M}$ be a set that is uniquely geodesically convex. Consider the distance $d(\cdot, \cdot)$ satisfying \cref{ass:equivalence_distance_exponential}.
Let us assume that the set $\mathcal{C}$ is a Riemannian $\alpha$-strongly convex set (\cref{def:riemannian_strong_convexity}). 
    The set $\mathcal{C}$ is then a geodesic $(\alpha\ell_{\M}/L_{\M}^2)$-strongly convex set (\cref{def:geodesic_strong_convexity}).
\end{restatable}

Before proving the theorem, we introduce the important notion of \textit{geodesic triangle}.
For the three points $p,q,r\in\mathcal{M}$, the set $\Delta pqr$ of the three minimizing geodesics joining these three points is a \textit{geodesic triangle}.
A \textit{comparison triangle} $\Delta\bar{p}\bar{q}\bar{r}$ is then a triangle with the same side length as $\Delta pqr$ in a metric space with a constant \textit{sectional curvature}.
\textit{Comparison theorems} are then used to compare the angle between these triangles according to a lower (Toponogov’s theorem) or an upper bound (Rauch's theorem) on the sectional curvature of the geodesic metric space $\mathcal{M}$.
We refer to \citep{burago1992ad,meyer1989toponogov,burago2001course} for a detailed treatment of such comparison theorems and to \citep{zhang2016first,zhang2018estimate} for their use in optimization contexts.
It should be noted that comparison theorems do not only compare triangles, but also \textit{hinges}, and result in angle or length comparisons \citep[Theorem 2.2. B]{meyer1989toponogov}.

\begin{proof}
As the set $\mathcal{C}$ is a Riemannian strongly convex set (\cref{def:riemannian_strong_convexity}), by using the definition of the strong convexity of sets in Hilbert spaces \eqref{eq:strong_convexity_euclidean}, we obtain
    \begin{align}\label{alg:str_besoin}
        & \forall x\in\mathcal{M},\;\;\forall p,\,q\in\operatorname{Exp}_x^{-1}(\mathcal{C}),\;\forall t\in[0,1],\nonumber\\
        & \;\text{if}\; z\in \TxM :\left\|z-\left(tp+(1-t)q\right)\right\|\leq \alpha t(1-t)\|p-q\|^2, \;\;\text{then} \;\; z\in\operatorname{Exp}_x^{-1}(\mathcal{C}),
    \end{align}
    for some parameter $\alpha > 0$. We now consider arbitrary points $ x, y \in\mathcal{M}$ and $\tilde z\in\mathcal{M}$ s.t. $d(\gamma(t),\tilde z) \leq \tilde{\alpha} t(1-t)d^2( x,  y)$, where $\gamma(t):[0, 1]\to\M$ is the geodesic between $ x$ and $y$ and $\tilde{\alpha} \defi \alpha\ell_{\M}L_{\M}^{-2} > 0$. Due to \cref{ass:equivalence_distance_exponential} we have  
    \begin{equation}\label{eq:our_assumption}
        \ell_{\M}\norm{\exponinv{\gamma(t)}(\tilde{z})} \leq L_{\M}^2\tilde{\alpha} t(1-t) \norm{\exponinv{x}(y)}^2.
    \end{equation}
    Now, \citep[Corollary 24]{martinez2022accelerated} using the Riemannian cosine law inequality for our Cartan--Hadamard manifold in the geodesic triangle with vertices $x$, $\tilde{z}$ and $\gamma(t)$, and the corresponding triangle in $T_x\M$ via $\exponinv{x}(\cdot)$, we have
\begin{align*}\label{alg:str_besoin}
    2\innp{\exponinv{x}(\tilde{z}), \exponinv{x}(\gamma(t))} &\geq \norm{\exponinv{x}(\tilde{z})}^2  + \norm{\exponinv{x}(\gamma(t))}^2- \norm{\exponinv{\gamma(t)}(\tilde{z})}^2  \\
    2\innp{\exponinv{x}(\tilde{z}), \exponinv{x}(\gamma(t))} &= \norm{\exponinv{x}(\tilde{z})}^2  + \norm{\exponinv{x}(\gamma(t))}^2 - \norm{\exponinv{x}(\tilde{z})- \exponinv{x}(\gamma(t))}^2,
\end{align*}
which implies 
\begin{equation}\label{eq:estimated_distance_is_less_in_hadamard}
    \norm{\operatorname{Exp}_x^{-1}(\tilde{z}) - \operatorname{Exp}_x^{-1}(\gamma(t))}_x \leq \norm{\operatorname{Exp}^{-1}_{\gamma(t)}(\tilde{z})}_{\gamma(t)}.
\end{equation}
    Now, let $p=\operatorname{Exp}^{-1}_{ x}( y), q=\operatorname{Exp}^{-1}_{ x}( x) = 0$, and $z = \operatorname{Exp}^{-1}_x(\tilde z)$, and note that $\operatorname{Exp}^{-1}_x(\gamma(t))= t \operatorname{Exp}_x^{-1}(y)$. Hence, combining \eqref{eq:our_assumption} with \eqref{eq:estimated_distance_is_less_in_hadamard} and using our new notation, we obtain
    \[
        \ell_{\M} \norm{z - (tp + (1-t)q)}_x \leq L_{\M}^2\tilde{\alpha} t (1-t) \norm{p - q}_x^2,
    \]
    and after using the value of $\tilde{\alpha} = \alpha\ell_{\M}L_{\M}^{-2}$ and \eqref{alg:str_besoin}, we conclude that $z = \operatorname{Exp}^{-1}_x(\tilde z)\in \operatorname{Exp}^{-1}_x(\mathcal{C})$, and hence $\tilde z\in\mathcal{C}$.
\end{proof}

By the same arguments presented in \cref{prop:riemannian_implies_geodesic}, one can establish that Riemannian strong convexity implies the geodesic strong convexity of sets in Cartan--Hadamard manifolds. This situation bears resemblance to the various notions of geodesically convex sets for Cartan--Hadamard manifolds \citep{boumal2020introduction}. It should be noted that geodesic and double geodesic strong convexity (\cref{def:geodesic_strong_convexity,def:double_geodesic_strong_convexity}) become equivalent under mild assumptions, which is noteworthy as \cref{def:geodesic_strong_convexity} relies on the geodesic metric space structure of $\mathcal{M}$, while \cref{def:double_geodesic_strong_convexity} leverages the manifold structure of $\mathcal{M}$.

\begin{restatable}[Equivalence between Geodesic and Double Geodesic Strong Convexity]{proposition}{equivalencedoublestrcvxhadamard}\label{prop:equivalence_double_str_cvx_hadamard}
Let $\mathcal{M}$ be a complete, connected Riemannian manifold, and let us assume that \cref{ass:equivalence_distance_exponential} holds.
If the subset $\mathcal{C}\subset\mathcal{M}$ is a geodesic $\alpha$-strongly convex set, then it is also a double geodesic $\frac{\alpha}{L_{\mathcal{M}}}$-strongly convex set. If the set $\mathcal{C}$ is a double geodesic $\alpha$-strongly convex set, then it is also a geodesic $\ell_{\mathcal{M}}\alpha$-strongly convex set.
\end{restatable}
\begin{proof}
($\Rightarrow$) We start with $\mathcal{C}\subset\mathcal{M}$ being an $\alpha$-geodesically strongly convex set. Now, let $ z\in \TgtM $ such that
\[
    \| z\| \leq \frac{\alpha}{L_{\mathcal{M}}} t(1-t)d^2(x,y).
\]
As the distance function $d$ can be bounded as $d(x,y)\leq L_{\mathcal{M}}\|\operatorname{Exp}^{-1}_{x}(y)\|$, we have 
\[
    \| z\| \leq \frac{\alpha}{L_{\mathcal{M}}} t(1-t)d^2(x,y) \quad \Rightarrow \quad d\left(\gamma(t),\operatorname{Exp}_{\gamma(t)}( z)\right)\leq \alpha t(1-t)d^2(x,y)
\]
As the set is geodesically strongly convex, we have $\operatorname{Exp}_{\gamma(t)}(z)\in\mathcal{C}$.

($\Leftarrow$) Now, we assume $\mathcal{C}$ to be a doubly exponentially strongly convex set with the parameter $\alpha$. We construct the point $z\in\mathcal{M}$ such that
\[
    d(\gamma(t),z) \leq \ell_{\mathcal{M}} \alpha t(1-t)d^2(x,y).
\]
    As the distance function $d$ can be bounded as $\ell_{\mathcal{M}}\|\operatorname{Exp}^{-1}_{\gamma(t)}(z)\| \leq d(\gamma(t),z)$, we have 
\[
    d(\gamma(t),z) \leq \ell_{\mathcal{M}} \alpha t(1-t)d^2(x,y) \quad \Rightarrow \quad \|\operatorname{Exp}^{-1}_{\gamma(t)}(z)\| \leq \alpha t(1-t)d^2(x,y)
\]
As the set is $\alpha$-double exponentially strongly convex, we have $z\in\mathcal{C}$.
\end{proof}

\section{Sublevel Sets of Geodesically Strongly Convex Functions}\label{sec:level_set_geodesically}

\cref{def:geodesic_strong_convexity,def:riemannian_strong_convexity,def:double_geodesic_strong_convexity} formalize the concept of strong convexity for subsets of a Riemannian manifold. Additionally, \cref{def:riemannian_scaling_inequality} serves as a key tool in \cref{sec:FW_geodesically_strongly_convex_set} to establish linear convergence. However, proving that a set is strongly convex might be difficult in practice. Hence, this section aims to develop the necessary framework for demonstrating the strong convexity of sets within Riemannian manifolds.

\subsection{Euclidean Case}
In the context of Hilbert spaces, the uniform convexity of $\ell_p$ or $p$-Schatten balls has been studied in the theory of Banach spaces \citep{clarkson1936uniformly,boas1940some,ball1994sharp}. However, for less standard cases, the strong convexity of sets can often be most efficiently demonstrated by showing that they correspond to the sublevel sets of strongly convex functions \citep{journee2010generalized,garber2015faster}.
We present the Riemannian counterpart of \citep[Theorem 12]{journee2010generalized} for sublevel sets of geodesically strongly convex functions (\cref{def:strongly_cvx_geo_fct}). These findings extend to the notion of geodesically convex functions, as observed in \citep{rapcsak2013smooth} and \citep[Proposition 11.8]{boumal2020introduction}.

In the Euclidean setting, \citep[Theorem 12]{journee2010generalized} demonstrates that the sublevel sets of $L$-smooth, $\mu$-strongly convex functions are strongly convex sets. In particular, the following set is $\mu/2\sqrt{2Ls}$ strongly convex,
\[
    Q_s := \{ x ~|~ f(x)-f^* \leq s \}, \quad \text{where} \;\; f^* = \min_x f(x).
\]

\subsection{Non-Euclidean Case} We now demonstrate that the sublevel sets of a geodesically smooth, strongly convex function are geodesic strongly convex sets. This result relies heavily on the following lemma.

\begin{restatable}[Smoothness property]{lemma}{Lsmoothnessbound}\label{lem:L_smoothness_bound}
Let us consider $f$ as a geodesically $L$-smooth function on the geodesically closed convex subset $\mathcal{C}\subset \mathcal{M}$, where $\mathcal{M}$ is an Cartan--Hadamard manifold. We denote $x^*\in\operatorname{argmin}_{x\in\mathcal{C}} f(x)$. Then,
\begin{equation}\label{eq:L_smoothness_bound}
    \|\nabla f(x)\|_x \leq \sqrt{2L(f(x) - f(x^*))}.
\end{equation}
\end{restatable}
This result is based on the concept of functional duality in a Riemannian manifold, which has been comprehensively studied in \citep{bergmann2021fenchel}. The corresponding proof can be referred to in \cref{ap:proofs_level_set}. 

\begin{restatable}[Geodesic Strong Convexity of sublevel sets]{theorem}{geodesicstrongconvexitylevelsets}\label{thm:geodesic_strong_convexity_level_sets} 
Let $\mathcal{M}$ be a Riemannian manifold. Suppose that $\mathcal{C}\subseteq\mathcal{M}$ is uniquely geodesically convex and $f:\mathcal{C}\subseteq\mathcal{M}\rightarrow \mathbb{R}$ is a proper, geodesically $L$-smooth, and $\mu$-strongly convex function on $\mathcal{C}$ w.r.t. the distance function $d$ satisfying \cref{ass:equivalence_distance_exponential}. Let $x^*\in\mathcal{C}$ satisfying $\nabla f(x^*) = 0$. Let $Q_s:= \big\{ x ~|~ f(x)-f^*\leq s \big\}\subseteq \mathcal{C}$ be geodesically strictly convex for some $s>0$, that is, every geodesic segment in $Q_s$ is in the interior of $Q_s$ except possibly for its endpoints.
Then, $Q_s$ is a geodesic strongly convex set with $\alpha = \mu/2\sqrt{2sL\max\{\ell^{-2}_\mathcal{M};\,1\}}$ (\cref{def:geodesic_strong_convexity}).
\end{restatable}

\begin{proof}
Let $s>0$, and let us consider $(x,y)\in Q_s^2$ and write $\gamma$ as the geodesic between $x$ and $y$. 
On successively using the geodesic smoothness of $f$, Cauchy-Schwartz, and \cref{lem:L_smoothness_bound}, for all smooth curve $c_{t}(\tilde t): c_t(0) = \gamma(t)$ and $t,\tilde t\in[0,1]$, we obtain
\begin{align}
f(c_t(\tilde t))-f^* \leq & f(\gamma(t))-f^* +  \langle \nabla f(\gamma(t)); \operatorname{Exp}^{-1}_{\gamma(t)}(c_t(\tilde t))\rangle_{\gamma(t)} + \frac{L}{2} d^2(\gamma(t), c_t(\tilde t)) \nonumber \\
\leq & f(\gamma(t))-f^* + \| \nabla f(\gamma(t))\|_{\gamma(t)} \|\operatorname{Exp}^{-1}_{\gamma(t)}(c_t(\tilde t))\|_{\gamma(t)} + \frac{L}{2} d^2(\gamma(t), c_t(\tilde t)) \nonumber\\
\leq &f(\gamma(t))-f^* + \sqrt{2L(f(\gamma(t))-f^*)} \ell^{-1}_{\mathcal{M}}d(\gamma(t), c_t(\tilde t)) + \frac{L}{2} d^2(\gamma(t), c_t(\tilde t)) \nonumber\\
\leq& \Big( \sqrt{f(\gamma(t))-f^*} + \sqrt{ \frac{L\max\{\ell^{-2}_{\mathcal{M}};\,1\}}{2} }  d(\gamma(t), c_t(\tilde t)) \Big)^2. \label{eq:condition_function_value}
\end{align}
Therefore, to ensure that $c_t(\tilde t)\in Q_s$, we can identify a sufficient condition on $d(\gamma(t), c_t(\tilde t))$, such that
\begin{align}
    & \Big( \sqrt{f(\gamma(t))-f^*} + \sqrt{ \frac{L\max\{\ell^{-2}_{\mathcal{M}};\,1\}}{2}} d(\gamma(t), c_t(\tilde t))\Big)^2 \leq s, \nonumber \\
    \Leftarrow &  d(\gamma(t), c_t(\tilde t)) \leq \text{sufficient condition} \leq \sqrt{\frac{2}{L\max\{\ell^{-2}_{\mathcal{M}};\,1\}}}\left(\sqrt{s} - \sqrt{f(\gamma(t))-f^*}\right) \label{eq:sufficient_condition_d}
 \end{align}
As $f$ is strongly convex, on using \cref{def:strongly_cvx_geo_fct} and because $(x,y)\in Q_s$, we obtain 
\[
    f(\gamma(t))-f^* \leq s - (1-t)t \frac{\mu}{2} d^2(x,y),
\]
Since $\sqrt{\cdot}$ is a concave function, we have $\sqrt{x-y} \leq \sqrt{x} - \frac{y}{2\sqrt{x}}$. Therefore,
\[
    \sqrt{f(\gamma(t))-f^*} \leq \sqrt{s - (1-t)t \frac{\mu d^2(x,y)}{2}} \leq \sqrt{s} - \frac{(1-t)t \mu d^2(x,y)}{4\sqrt{s}}.
\]
Hence, we have
\begin{align}
      \sqrt{\frac{2}{L\max\{\ell^{-2}_{\mathcal{M}};\,1\}}}\left(\sqrt{s} - \sqrt{f(\gamma(t))-f^*}\right) \geq & \sqrt{\frac{2}{L\max\{\ell^{-2}_{\mathcal{M}};\,1\}}}\left(\frac{(1-t)t \mu d^2(x,y)}{4\sqrt{s}}\right) \\
      = & (1-t)t  \frac{\mu}{2\sqrt{2sL\max\{\ell^{-2}_{\mathcal{M}};\,1\}}} d^2(x,y).\label{eq:relation_dsquare}
\end{align}
Therefore,
\begin{align*}
    & d(\gamma(t), c_t(\tilde t)) \leq  (1-t)t  \frac{\mu}{2\sqrt{2sL\max\{\ell^{-2}_{\mathcal{M}};\,1\}}} d^2(x,y) \\
    \xRightarrow{\eqref{eq:relation_dsquare}} \;\; & d(\gamma(t), c_t(\tilde t)) \leq \sqrt{\frac{2}{L\max\{\ell^{-2}_{\mathcal{M}};\,1\}}}\left(\sqrt{s} - \sqrt{f(\gamma(t))-f^*}\right)  \\
    \xRightarrow{\eqref{eq:sufficient_condition_d}} \;\; & \Big( \sqrt{f(\gamma(t))-f^*} + \sqrt{ \frac{L\max\{\ell^{-2}_{\mathcal{M}};\,1\}}{2}} d(\gamma(t), c_t(\tilde t))\Big)^2 \leq s \\
    \xRightarrow{\eqref{eq:condition_function_value}} \;\; & f(c_t(\tilde t))-f^*\leq s.
\end{align*}
Hence, $c_t(\tilde t)$ is in the set $Q_s$, which is the definition of geodesic strong convexity.
\end{proof}

\textbf{Example: Unit sphere.} Let us consider $\mathbb{S}^{n-1}$, the unit sphere manifold embedded in $\mathbb{R}^n$, with the distance function $d(x,y) = \arccos(\langle x;\;y\rangle)$. Let us fix $x_0\in\mathcal{M}$, and let $f(x)=d^2(x_0,x)$. Let the set $Q_s := \{x:f(x)\leq s\}$. When $s<\left(\frac{\pi}{2}\right)^2$, the squared distance function is a geodesically smooth and strongly convex function (the constants of which depend on $s$) \citep[Lem. 12.15]{lee2018introduction}, \citep[pp153–154]{sakai1996riemannian}. As $Q_s$ is also a strictly convex set for $s<(\pi/2)^2$, the set $Q_s$ is a geodesically strongly convex set, as shown in \cref{thm:geodesic_strong_convexity_level_sets}.

\textbf{Example: Symmetric Positive Definite Matrices.} Let $\mathcal{M}$ be the set of symmetric positive definite matrices with the affine-invariance metric, which yields the distance function $d_\mathcal{M}(X,Y) =\sqrt{ \sum_i \log^2 \lambda_i(X^{-1}Y)) }$ and is a Cartan--Hadamard manifold, cf. \citep{hosseini2015matrix}. 
Let us fix $X_0\in\mathcal{M}$, and let $f(X)=d_\mathcal{M}^2(X_0,X)$. As $\mathcal{M}$ is a Cartan--Hadamard manifold, its distance function is strongly convex, and $d$ is also smooth in bounded sets. Therefore, the sets $Q_s$, $s<f(0)$, are strongly convex, as shown in \cref{thm:geodesic_strong_convexity_level_sets}.

\section{Double Geodesic Strong Convexity and Approximate Riemannian Scaling Inequality} \label{sec:double_geodesic_convexity}

Under mild assumptions, \cref{sec:FW_geodesically_strongly_convex_set} shows that the Riemannian FW algorithm (\cref{algo:RFW}) admits a global linear convergence rate (\cref{th:RFW_riemannian_scaling_inequality}) when the feasible set in \eqref{eq:riemannian_manifold_opt} satisfies a Riemannian scaling inequality.
However, there is no apparent link between (double) geodesic strong convexity (\cref{def:riemannian_strong_convexity} or \cref{def:double_geodesic_strong_convexity}) and the Riemannian scaling inequality. Therefore, in this section, we explore the link between double geodesic strong convexity and the Riemannian scaling inequality. We introduce the notion of \textit{approximate Riemannian scaling inequality} (\cref{def:approximate_riemannian_scaling_inequality}) and demonstrate that the quality of the approximation depends on the \textit{exponential map operator} (\cref{def:double_exponential_map}).

\subsection{Double Geodesic Strong Convexity and Double Exponential Map}

In this section, we link the double geodesic strong convexity (\cref{def:double_geodesic_strong_convexity}) with an approximate version of the scaling inequality (introduced subsequently in \cref{def:approximate_riemannian_scaling_inequality}). We first introduce the \textit{double exponential map} \citep{gavrilov2007double,dzhepko2008double,nikonorov2013double} and rewrite the definition of the double geodesic strong convexity using two geodesic paths.

\begin{definition}[Double Exponential Map]\label{def:double_exponential_map}
Let $\mathcal{M}$ be a complete, connected Riemannian manifold. Let $\TxM $ be the tangent space to $\mathcal{M}$ at $x\in\mathcal{M}$, $\operatorname{Exp}_x(\cdot):\TxM \rightarrow\mathcal{M}$ be the exponential map at $x$, and $\Gamma_x^y: \TxM \rightarrow \TyM $ be the transportation map between the tangent spaces $\TxM $ and $\TyM $.
We define the \textit{double exponential map} at $x\in\mathcal{M}$ as the function $ \operatorname{Exp}_x(u,v):\TxM \times \TxM \rightarrow \mathcal{M}$, such that
\begin{equation}\label{eq:double_exponential_map}\tag{double exponential map}
\operatorname{Exp}_x(u,v) := \operatorname{Exp}_{\operatorname{Exp}_x(u)}\big(\Gamma_x^{\operatorname{Exp}_x(u)} v\big).
\end{equation}
We also define $h_x(\cdot,\cdot):\TxM \times \TxM  \rightarrow \TxM $ as the (unique) exponential map operator, such that
\begin{equation}\label{eq:exponential_map_operator}\tag{Exponential Map Operator}
\operatorname{Exp}_{x}(h_x(u,v)) := \operatorname{Exp}_x(u,v),~~\forall u,v\in \TxM .
\end{equation}
\end{definition}

In particular, we can rewrite the double geodesic strong convexity of a set with the double exponential map. Informally, \cref{def:double_geodesic_strong_convexity} takes into consideration the geodesic $\gamma$ between $x$ and $y$ and other geodesics departing from a $\gamma(t)$ that moves in every $z$ direction, thus describing a closed ball in the tangent space $\TgtM $. Therefore, \eqref{eq:strong_convexity_riemannian} in \cref{def:double_geodesic_strong_convexity} becomes
\begin{equation}\label{eq:geodesic_strong_convexity_double_exponential}
    \operatorname{Exp}_x\Big(\operatorname{Exp}_x^{-1}(\gamma(t)), \alpha t (1-t) d^2(x,y) z\Big) \in \mathcal{C}, \text{ for all } z \in \bar{B}(0, 1) \subset T_x\mathcal{M}.
\end{equation}
This expression motivates the use of the term \textit{double} to describe the notion of strong convexity.

\subsection{Link with Approximate Riemannian Scaling Inequality}

This section presents a connection between double geodesic strong convexity and a weaker version of the Riemannian scaling inequality using the exponential map operator. When the exponential map operator satisfies $h(u,v)=u+v$ (basically, when the set is Euclidean), the double geodesic set strong convexity implies the Riemannian scaling inequality (\cref{prop:implication_riemannian_scaling_inequality}).
Instead, in \cite{dzhepko2008double,nikonorov2013double}, explicit approximations of the \textit{exponential map operator} were proposed when the Riemannian manifold is symmetric or has a constant curvature, e.g., the Euclidean sphere or Lobachevsky spaces \citep{dzhepko2008double}.
These approximations provide expansions of $h_x(u,v)$ of the form
\begin{equation}\label{eq:approximation_exponential_operator}
h_x(u,v) = u + v + R_x(u,v),
\end{equation}
where the term $R_x(u,v)$ can be an order of magnitude smaller than $\sqrt{\|u\|_x^2 +\|v\|_x^2}$.
In these cases, we can no longer prove that the double geodesic strong convexity would imply the Riemannian scaling inequality.
Instead, we introduce an \textit{approximate Riemannian scaling inequality} in \cref{def:approximate_riemannian_scaling_inequality} and demonstrate in \cref{prop:implication_approximate_riemannian_scaling_inequality} that the double geodesic strong convexity implies an approximate Riemannian scaling inequality. The approximation quality is controlled by $R_x(u,v)$.

\begin{definition}[Approximate Riemannian Scaling Inequality]\label{def:approximate_riemannian_scaling_inequality}
Let $\mathcal{M}$ be a Riemannian manifold and $\mathcal{C}\subset\mathcal{M}$.
Let us consider a geodesically $\alpha$-strongly convex set $\mathcal{C}\subset\mathcal{M}$.
We then say that $\mathcal{C}$ satisfies the approximate Riemannian scaling inequality w.r.t. the distance $d(\cdot,\cdot)$ and the residual $r(\cdot):\mathcal{C}\rightarrow T\mathcal{M}$ if, for all $x\in\mathcal{C}$, $w\in \TxM $, and $v \in \operatorname{argmax}_{z\in\mathcal{C}} \langle w; \operatorname{Exp}^{-1}_x(z)\rangle_x$, we have
\begin{equation}\label{eq:approximate_scaling_inequality_riemannian}\tag{Approximate Riemannian Scaling Inequality}
\langle w; \operatorname{Exp}^{-1}_x(v)\rangle_x \geq \alpha \|w\|_x d(v, x)^2 + \langle w; r(x)\rangle_x,
\end{equation}
\end{definition}

In \cref{prop:implication_approximate_riemannian_scaling_inequality}, we now show that the geodesic strong convexity of a set in the Riemannian manifold $\mathcal{M}$ implies an approximate Riemannian scaling inequality that depends on this difference $R_x(u,v)$ of the exponential operator map.

\begin{restatable}[Double Geodesic Str. Cvx. implies Approximate Riemannian Scaling Inequality]{proposition}{strcvximpliesapproximateriemannianscaling}\label{prop:implication_approximate_riemannian_scaling_inequality}
Let us consider $\mathcal{C}\subset\mathcal{M}$ as double geodesically $\alpha$-strongly convex (\cref{def:double_geodesic_strong_convexity}) in a Riemannian manifold $\mathcal{M}$.
We define $R_x:\TxM \times \TxM \rightarrow \TxM $, such that, for all $x\in\mathcal{C}$, the exponential map operator (\cref{def:double_exponential_map}) is decomposed as
\begin{equation}\label{eq:residual_condition_exponential_operator_map}
h_x(u,v) = u + v + R_x(u,v),~~\forall(u,v)\in \TxM ,~x\in\mathcal{C}.
\end{equation}
The approximate scaling inequality (\cref{def:approximate_riemannian_scaling_inequality}) is then satisfied w.r.t. $d(\cdot,\cdot)$ with the residual $r(\cdot)$ s.t. for all $x\in\mathcal{C}$ and $w\in \TxM $ 
\begin{equation}\label{eq:residual_geodesic_strong_convexity}\tag{Residual}
r(x) = R_x\Big[\frac{1}{2}\gamma^\prime_{x,v}(0), \Gamma^x_{\gamma_{x,v}(1/2)}\Big(\frac{\alpha}{4}d^2(x,v)z^*\Big) \Big],
\end{equation}
where $v \in \operatorname{argmax}_{z\in\mathcal{C}} \langle w; \operatorname{Exp}^{-1}_x(z)\rangle_x$ and $z^*\in\operatorname{argmax}_{\|z\|_{\gamma_{x,v}(1/2)}=1} \big\langle \Gamma_x^{\gamma_{x,v}(1/2)} w; z \big\rangle_{\gamma_{x,v}(1/2)}$.
\end{restatable}

\begin{proof}
This proof is similar to \cref{prop:implication_riemannian_scaling_inequality} until \eqref{eq:optimality_inequality}.
We then write 
\[
u = \frac{1}{2}\gamma^\prime_{x,v}(0) \text{ and } \omega =\big(\Gamma_x^{\gamma_{x,v}(1/2)}\big)^{-1}(\frac{\alpha}{4}d^2(x,v)z).
\]
Hence, we have
\[
\Big\langle w; \operatorname{Exp}^{-1}_x(v)\Big\rangle_x \geq \Big\langle w; \operatorname{Exp}^{-1}_x\Big(\operatorname{Exp}_{x}\big(u + \omega + R(u,v)\big)\Big)\Big\rangle_x.
\]
Using the same arguments as those in the proof of \cref{prop:implication_riemannian_scaling_inequality}, we obtain
\[
\Big\langle w; \operatorname{Exp}^{-1}_x(v)\Big\rangle_x \geq \frac{\alpha}{2} d^2(x,v) \|w\|_x + \Big\langle w; R(\frac{1}{2}\gamma^\prime_{x,v}(0), \Gamma^x_{\gamma_{x,v}(1/2)}(\frac{\alpha}{4}d^2(x,v)z^*)\Big\rangle_x.
\]
\end{proof}

The approximate scaling inequality is meaningful in the regime wherein $x$ and $v$ are close. 
In this situation, the scale of the residual $r(\cdot)$ is determined by  $\gamma^\prime_{x,v}(0)$ and $d^2(x,v)$.
In some setting, the residual is such that the approximate term $\langle w; r(x)\rangle_x$ in \cref{def:approximate_riemannian_scaling_inequality} becomes negligible w.r.t. the original term $\alpha \|w\|_x d(x,v)^2$.
It should be noted that, in the analysis of the FW algorithm in \cref{sec:FW_geodesically_strongly_convex_set}, we select $w:=-\nabla f(x_t)$, where $x_t$ are the FW iterates.

The general expression of $R_x(u,v) = h_x(u,v) - u - v$ as a series is given in \citep{gavrilov2007double}.
\citet{dzhepko2008double} provide the explicit Taylor series for the Euclidean sphere \citep[(6)]{dzhepko2008double} and Lobachevskii spaces \cite[Section 3]{dzhepko2008double}.

\section{Frank-Wolfe on Geodesically Strongly Convex Sets}\label{sec:FW_geodesically_strongly_convex_set}

The FW algorithm is a first-order method that is used to solve constrained optimization problems in Banach spaces.
Each iteration relies on a linear minimization step over the constraint region.
It has recently been extended for constrained optimization over Riemannian manifolds.
The Riemannian Frank-Wolfe (RFW) algorithm \citep{weber2017riemannian,weber2019projection} solves the following smooth convex constrained problem.
\begin{equation}\label{eq:riemannian_manifold_opt}\tag{OPT}
{\text{minimize }} f(x), \quad \text{for} \;\;x\in\mathcal{C}
\end{equation}
where $\mathcal{C}\subseteq\mathcal{M}$ is a compact geodesically convex set, and $f$ is geodesically smooth and convex.

\begin{algorithm}
\begin{algorithmic}
\caption{Riemannian Frank-Wolfe (RFW) algorithm}\label{algo:RFW}
\Require $x_0\in\mathcal{C}\subset\mathcal{M}$; assume access to the geodesic map $\gamma:[0,1]\rightarrow \mathcal{M}$\;
\For{ $t=0,1,\cdots$}
    \State $\qquad \textbf{1.} \; $ $v_t \leftarrow \operatorname{argmax}_{v\in\mathcal{C}}\langle \nabla f(x_t); \operatorname{Exp}^{-1}_{x_t}(v)\rangle$\; 
    \State $\qquad \textbf{2.} \; $ Let $s_t = \operatorname{argmin}_{s\in[0,1]} s \langle \nabla f(x_t); \operatorname{Exp}^{-1}_{x_t}(v_t)\rangle_x + s^2\frac{L}{2} d^2(x_t, v_t)$. \;\\ 
    \State $\qquad \textbf{3.} \; $ $x_{t+1} \leftarrow \gamma(s_t)$, where $\gamma(0)=x_t$ and $\gamma(1)=v_t$.\;
    \EndFor
\end{algorithmic}
\end{algorithm}

Remarkably, \cite{weber2017riemannian,weber2019projection} proves similar convergence rates of RFW as the FW algorithm with comparable structural assumptions on the optimization problem as in the Hilbertian setting. 
For instance, when the function $f$ is geodesically convex and smooth and the set is compact convex, the RFW algorithm converges in $\mathcal{O}(1/T)$ \citep[Theorem 3.4.]{weber2017riemannian}.
Similarly, \citet[Theorem 3.5.]{weber2017riemannian} show linear convergence of \cref{algo:RFW} when using short-step sizes, and the objective function is geodesically strongly convex and the optimum in the interior of the set.

In the Hilbertian setting, the FW algorithm admits various accelerated convergence regimes when the set is strongly convex. 
When the unconstrained optimum of $f$ is outside the constraint set, the FW algorithm converges linearly \citep{demyanov1970}, or when the function is strongly convex, the convergence is in $\mathcal{O}(1/T^2)$ without an assumption on the unconstrained optimum location \citep{garber2015faster}.
The previous sections establish possible notions of strong convexity for sets in Riemannian manifolds.
We now demonstrate that analog convergence regimes for the RFW algorithm on geodesically convex sets hold as in the case of the Hilbertian setting. These results complete the work of \cite{weber2017riemannian,weber2019projection}.

\subsection{Linear convergence of RFW under the Riemannian scaling inequality}

As outlined in \citep{kerdreux2021projection}, the \textit{scaling inequality} (\cref{thm:equivalence_euclidean}.(b)) is a convenient characterization of the strong convexity of the set (in Hilbertian setting) for establishing convergence rates.
We follow the same path and establish the convergence rate of the RFW algorithm when the constraint set satisfies (approximate) Riemannian scaling inequalities.
In \cref{th:RFW_riemannian_scaling_inequality}, we thus prove the linear convergence of the RFW algorithm when the set $\mathcal{C}$ satisfies a Riemannian scaling inequality and the unconstrained optimum of $f$ is outside $\mathcal{C}$.
This provides a generalization of \cite{demyanov1970} in a Riemannian setting.

\begin{restatable}[Linearly Convergent RFW with Riemannian Scaling Inequality]{theorem}{RFWriemannianscalinginequality}\label{th:RFW_riemannian_scaling_inequality}
    Consider a geodesically convex set $\mathcal{C}$ and a geodesically convex function $f$ that is $L$-smooth in $\mathcal{C}$. Assume that any unconstrained optimum of $f$ lies outside the constraint set $\mathcal{C}$, and in particular, there exists $c>0$ s.t. $\operatorname{argmin}_{x\in\mathcal{C}}\|\nabla f(x)\|_x>c$.
Assume that, for every $x\in\mathcal{C}$, the \eqref{eq:scaling_inequality_riemannian} holds.
Then \cref{algo:RFW} converges linearly:
\[
f(x_{t+1}) - f(x^*) \leq (f(x_{t}) - f(x^*)) \max\{1/2, 1- \alpha c/(2L)\}.
\]
\end{restatable}
\begin{proof}
This proof is based on \citep{demyanov1970,garber2015faster,kerdreux2021projection}, but in in a Riemannian setting.
Using the geodesic smoothness (\cref{def:geodesic_smoothness_fct}) of $f$ at $x_{t+1}=\gamma(s_t)$, we obtain
\[
f(x_{t+1}) \leq f(x_t) - \big\langle -\nabla f(x_t); \operatorname{Exp}^{-1}_{x_t}(\gamma(s_t)) \big\rangle_x + \frac{L}{2} d^2(x_t,\gamma(s_t)).
\]
As $\gamma$ is a geodesic between $x_t$ and $v_t$, we have $d(x_t,\gamma(s_t)) = s_t d(x_t, v_t)$ and $\operatorname{Exp}^{-1}_{x_t}(\gamma(s_t))= s_t \operatorname{Exp}^{-1}_{x_t}(v_t)$.
Hence, we now have
\[
f(x_{t+1}) \leq f(x_t) - s_t \big\langle -\nabla f(x_t); \operatorname{Exp}^{-1}_{x_t}(v_t) \big\rangle_x + \frac{L}{2} s_t^2 d^2(x_t,v_t).
\]
According to the \text{short-step} rule for $s_t$ (\cref{algo:RFW}), for all $s\in[0,1]$, we now have 
\[
f(x_{t+1}) \leq f(x_t) - s \big\langle -\nabla f(x_t); \operatorname{Exp}^{-1}_{x_t}(v_t) \big\rangle_x + \frac{L}{2} s^2 d^2(x_t,v_t).
\]
After using the optimality of $v_t$, we have $ \langle -\nabla f(x_t); \operatorname{Exp}^{-1}_{x_t}(v_t) \rangle_x \leq \langle -\nabla f(x_t); \operatorname{Exp}^{-1}_{x_t}(x^*) \rangle_x$, where $x^*\in\mathcal{C}$ is a solution to \eqref{eq:riemannian_manifold_opt}.
Then, owing to the geodesic convexity of $f$, we have 
\[
f(x^*) - f(x) \geq \langle \nabla f(x); \operatorname{Exp}_x^{-1}(x^*) \rangle_x.
\]
Hence, as it is the case in the Hilbertian setting, the FW gap $\langle -\nabla f(x_t); \operatorname{Exp}^{-1}_{x_t}(v_t) \rangle_x$ upper bounds the primal gap at $x_t$, i.e.,
\[
f(x_t) - f(x^*) \leq \langle -\nabla f(x_t); \operatorname{Exp}^{-1}_{x_t}(v_t) \rangle_x.
\]
We write $h_t = f(x_t) - f(x^*)$, and we hence have
\begin{equation}\label{eq:inter_proof_linear}
h_{t+1} \leq h_t (1 - s/2) - s/ 2\big\langle -\nabla f(x_t); \operatorname{Exp}^{-1}_{x_t}(v_t) \big\rangle_x + \frac{L}{2} s^2 d^2(x_t,v_t).
\end{equation}
Now, with $c = \operatorname{argmin}_{x\in\mathcal{C}}\|\nabla f(x)\|_x>0$, after using the \eqref{eq:scaling_inequality_riemannian} at $x_t$ and since $-\nabla f(x_t)\in T_{x_t}\mathcal{M}$, we obtain
\[
\langle -\nabla f(x_t); \operatorname{Exp}^{-1}_{x_t}(v_t) \rangle_x \geq \alpha c d(x_t, v_t)^2,
\]
such that, for all $s\in[0,1]$, we have
\begin{equation}\label{eq:intermediaire}
h_{t+1} \leq h_t (1 - s/2) +  \frac{s}{2}\Big(Ls - \alpha c\Big) d^2(x_t,v_t).
\end{equation}
Then, if $\alpha c/L <1$, by choosing $s=\alpha c/L$ in \eqref{eq:intermediaire}, we have $h_{t+1}\leq h_t (1 - \alpha c/(2L))$; else, we have $L-\alpha c < 0$, and on selecting $s=1$, we simply have $h_{t+1}\leq h_t/2$. Hence,
\[
h_{t+1} \leq h_t \operatorname{max}\{1/2, 1- \alpha c/(2L)\}.
\]
\end{proof}

Examples of sets satisfying the Riemannian scaling inequality are, for instance, sets of restricted diameter in Riemannian manifolds with bounded sectional curvature (see \cref{ex:riemannian_strongly_convex_sets}). However, this condition might be restrictive.

\subsection{Local Linear convergence of RFW under the approximate Riemannian scaling inequality}

Since the ``exact'' Riemannian scaling inequality is quite restrictive,  this section provides a similar convergence result when the feasible sets satisfy only approximate Riemannian scaling inequalities (\cref{def:approximate_riemannian_scaling_inequality}).

\begin{restatable}[Linearly Convergent RFW on Double Geodesic Strongly Convex Sets]{theorem}{RFWgeodesicalstrongconvexity}\label{th:RFW_geodesical_strong_convexity}
    Consider a complete connected Riemannian manifold and a distance $d(\cdot, \cdot)$ satisfying \cref{ass:equivalence_distance_exponential} with parameters $0 < \ell_{\mathcal{M}}\leq L_{\mathcal{M}}$.
    Assume that $\mathcal{C}\subset\mathcal{M}$ is an $\alpha-$double geodesically strong convex set w.r.t. the distance $d(\cdot,\cdot)$ (\cref{def:double_geodesic_strong_convexity}). Assume that the function $f$ is a geodesically convex $L$-smooth function, and there exists $c>0$ s.t. 
    \[
        \operatorname{min}_{x\in\mathcal{C}}\|\nabla f(x)\|_x>c.
    \]
    
   Let us assume that there exists $C>0$ s.t. that the residual \eqref{eq:residual_condition_exponential_operator_map} of the exponential map operator $R_x(\cdot,\cdot)$ for all $x\in\mathcal{C}$ and $u,w \in \TxM $ is such that
    \begin{equation}\label{eq:growth_condition_residual}
    \|R_x(u,w)\|_x \leq C \cdot \operatorname{max}\big\{\|u\|_x^2\|w\|_x; \|w\|_x^2\|u\|_x\big\}.
    \end{equation}
   Let us assume that, for some diameter $\delta>0$, $d(x_t,v_t)^2\leq (\alpha c)/(2\delta L\tilde{C})$, where $\tilde{C}:= C~\operatorname{max}\big\{ \frac{\alpha^2}{{16 \delta \ell_{\mathcal{M}}}}; \frac{\alpha}{4 \ell^2_{\mathcal{M}}}\big\}$.
    Then,
    \[
    f(x_{t+1}) - f(x^*) \leq (f(x_{t}) - f(x^*)) \max\left\{\frac{1}{2}, 1- \frac{\alpha c}{2L}\right\}.
    \]
\end{restatable}
\begin{proof}
As in the proof of Theorem \eqref{th:RFW_riemannian_scaling_inequality}, \eqref{eq:inter_proof_linear} is satisfied, i.e.,
\begin{equation}\label{eq:intermediary2}
h_{t+1} \leq h_t (1 - s_t/2) - s_t/ 2\langle -\nabla f(x_t); \operatorname{Exp}^{-1}_{x_t}(v_t) \rangle_{x_t} + \frac{L}{2} s_t^2 d^2(x_t,v_t).
\end{equation}
Now, from \cref{prop:implication_approximate_riemannian_scaling_inequality}, as $\mathcal{C}$ is double geodesically $\alpha$-strongly convex, an approximate Riemannian scaling inequality is satisfied (\cref{def:approximate_riemannian_scaling_inequality}) at $x_t$ with $-\nabla f(x_t)\in T_{x_t}\mathcal{M}$ with the residual $r(x_t)$ as in \eqref{eq:residual_geodesic_strong_convexity}, i.e.,
\begin{equation}\label{eq:residual_intermediary}
r(x_t) = R_{x_t}\Big[\frac{1}{2}\gamma^\prime_{x_t,v_t}(0), \Gamma^{x_t}_{\gamma_{x_t,v_t}(1/2)}\Big(\frac{\alpha}{4}d^2(x_t,v_t)z^*\Big) \Big].
\end{equation}
Hence, on combining with $\operatorname{argmin}_{x\in\mathcal{C}}\|\nabla f(x)\|_x>c$, we can lower-bound $\langle -\nabla f(x_t); \operatorname{Exp}^{-1}_{x_t}(v_t) \rangle_{x_t}$ as follows:
\[
\langle -\nabla f(x_t); \operatorname{Exp}^{-1}_{x_t}(v_t) \rangle_{x_t} \geq \alpha c d(x_t, v_t)^2 + \langle -\nabla f(x); r(x_t)\rangle_{x_t}.
\]
On substituting this inequality in \eqref{eq:intermediary2}, for all $s\in[0,1]$, we have
\begin{equation}\label{eq:intermediary_second}
h_{t+1} \leq h_t (1 - s/2) +  \frac{s}{2}\Big(L s - \alpha c\Big) d^2(x_t,v_t) + \frac{s}{2}\langle \nabla f(x_t); r(x_t)\rangle_{x_t}.
\end{equation}
We use \eqref{eq:growth_condition_residual} to upper-bound the term $\|r(x_t)\|_{x_t}$.
Hence, we are first required to obtain an upper bound on $\|\gamma^\prime_{x_t, v_t}(0)\|_{x_t}$ and $\|\frac{\alpha}{4}d^2(x_t,v_t) z^*\|_{\gamma_{x_t,v_t}(1/2)}$.
We first note that
\[
\big\|\frac{\alpha}{4}d^2(x_t,v_t) z^*\big\|_{\gamma_{x_t,v_t}(1/2)} = \frac{\alpha}{4}d^2(x_t,v_t).
\]
Furthermore, by definition of the exponential mapping, we have $\|\gamma^\prime_{x_t, v_t}(0)\|_{x_t}=\|\operatorname{Exp}_{x_t}^{-1}(v_t)\|_x$, and according to \cref{ass:equivalence_distance_exponential}, we have
\[
\|\gamma^\prime_{x_t, v_t}(0)\|_{x_t} \leq \frac{1}{\ell_{\mathcal{M}}} d(x_t, v_t).
\]
When plugging these two bounds in the residual $r(x_t)$ \eqref{eq:residual_intermediary}, with the growth condition on the residual \eqref{eq:growth_condition_residual}, we have
\begin{eqnarray}
\|r(x_t)\|_{x_t} & \leq & C ~ \operatorname{max}\big\{ \frac{\alpha^2}{16 \ell_{\mathcal{M}}} d^5(x_t,v_t); \frac{\alpha}{4 \ell^2_{\mathcal{M}}} d^4(x_t,v_t)\big\}\\
\|r(x_t)\|_{x_t} & \leq & C d^4(x_t,v_t) ~ \operatorname{max}\big\{ \frac{\alpha^2}{16 \ell_{\mathcal{M}}} d(x_t,v_t); \frac{\alpha}{4 \ell^2_{\mathcal{M}}}\big\}.
\end{eqnarray}
On using $d(x_t,v_t)\leq \delta$, and with $\tilde{C}:= C~\operatorname{max}\big\{ \alpha^2/{(16 \ell_{\mathcal{M}})} \delta; \alpha/(4 \ell^2_{\mathcal{M}})\big\}$, we have $\|r(x_t)\|_{x_t}\leq \tilde{C} d^4(x_t,v_t)$.
Hence, using the Cauchy-Schwartz inequality and $\|\nabla f(x_t)\|_x\leq \delta L$, \eqref{eq:intermediary_second} becomes
\begin{eqnarray}
h_{t+1} & \leq & h_t (1 - s/2) +  \frac{s}{2}\Big(L s - \alpha c\Big) d^2(x_t,v_t) + \frac{s}{2}\delta L \tilde{C} d^4(x_t,v_t)\\
 & \leq & h_t (1 - s/2) +  \frac{s}{2} d^2(x_t,v_t) \Big(L s - \alpha c + \delta L \tilde{C} d^2(x_t,v_t)\Big).
\end{eqnarray}
As we assumed $d(x_t,v_t)^2\leq (\alpha c)/(2\delta L\tilde{C})$, we have $- \alpha c + \delta L \tilde{C} d^2(x_t,v_t) \leq -(\alpha c)/2 < 0$.
Let us consider $s^*:= (\alpha c - \delta L \tilde{C}d^2(x_t,v_t))/L > 0$.
If $s^*>1$, then the choice of $s=1$ results in $h_{t+1}\leq h_t/2$.
Else, $s^*\in[0,1]$, and we select $s=s^*$. We hence obtain $h_{t+1} \leq h_t(1-s^*/2) \leq 1 - (\alpha c)/(2L)$.
Overall, we obtain
\[
h_{t+1} \leq h_t \operatorname{max}\{1/2, 1- \alpha c/(2L)\}.
\]
\end{proof}

The previous theorem thus states that, provided a burn-in phase (that follows from the general $\mathcal{O}(1/T)$ of the RFW algorithm) to ensure $d(x_t,v_t)^2\leq (\alpha c)/(2\delta L\tilde{C})$, the iterates of the RFW algorithm converge linearly when the set satisfies the approximate Riemannian scaling inequality.

\begin{example}[Riemannian trust-region subproblem.]  When minimizing a function $f$ on a Riemannian manifold, a common approach is to use the trust-region method. After $t$ iterations, it approximates $\tilde f_t\approx f$, where $\tilde f_t$ is usually a second-order approximation of $f$ around $x_{t}$. Common trust region approaches minimize $\tilde f_t(x_t+\Delta)$, where $\Delta$ belongs to the tangent space of $x_t$ under the constraint that $\|\Delta\|\leq \delta_t$; we then use a retraction on $x_t+\Delta$ to obtain the iterate $x_{t+1}$. Alternatively, it is also possible to solve the subproblem directly on the manifold as follows:
\[
    x_{t+1} = \argmin_{x\in\mathcal{M},\;d(x_t,\,x)\leq \delta_t} \tilde f_t(x).
\]
On using \cref{thm:geodesic_strong_convexity_level_sets}, we determine that the set is geodesically strongly convex. Therefore, if the set is sufficiently small, the RFW algorithm converges linearly on this subproblem.
\end{example}

\begin{example}[Global Riemannian optimization through local subproblem solving]
    For manifolds of curvature bounded in $[\kmin, \kmax]$, and defining $K \defi \max \{\abs{\kmin}, \kmax\}$, \citet{martinez2020global} presented a reduction from global Riemannian g-convex optimization to optimization in Riemannian balls of radius $\bigo{\frac{1}{\sqrt{K}}}$. It is required to solve $\bigo{\zeta_{R}}$ of such ball optimization problems to an accuracy proportional to the final global accuracy times low polynomial factors on other parameters. Here, $R$ is the initial distance to an optimizer, and $\zeta_{R}$ is a natural geometric constant, cf. \cref{prop:dist_func_strong_cvx_smooth}. If we can solve the subproblems using linear rates, the reduction only adds a $\bigotilde{\zeta_R}$ factor to these linear rates.
\end{example}

\subsection{Numerical Experiment: Minimization Over a Sphere}

\begin{figure}[t]
    \centering
    \includegraphics[width=0.60\textwidth]{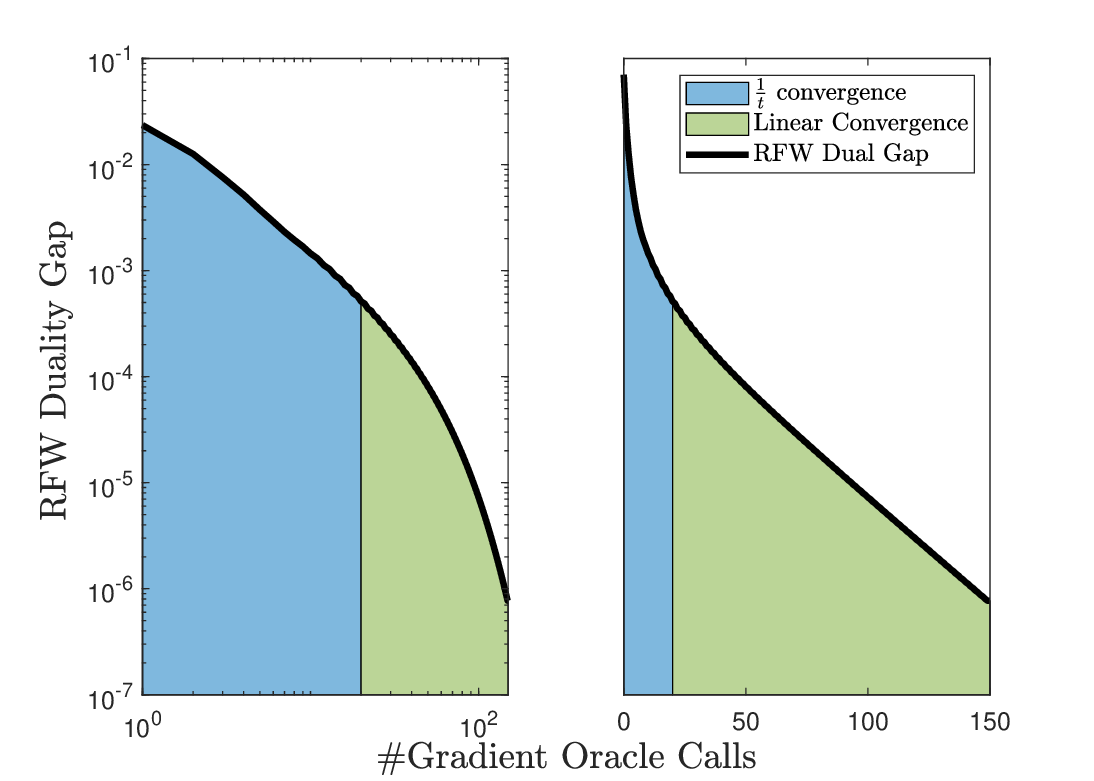}
    \caption{Numerical convergence of the RFW algorithm's iterates (Algo. \ref{algo:RFW}) for minimizing a quadratic function over a double geodesic strongly convex set in a unit sphere. The dimension of the problem is $n=500$, $x_c$ is a vector of ones, and the function $f$ is a random quadratic function parametrized as $f(x)\triangleq \|A(x-x^\star)\|^2$, where $A$ is a random $250\times 500$ matrix, and $x^\star$ is generated at random such that $\operatorname{dist}(x^\star,x_c)\leq \pi/2$. The parameter $R$ is set such that $R = 0.9\operatorname{dist}(x^\star,x_c)$, which ensures that the solution lies on the boundary, and therefore, there exists $c>0$ s.t. $\operatorname{argmin}_{x\in\mathcal{C}}\|\nabla f(x)\|_x>c$. As predicted by \cref{th:RFW_geodesical_strong_convexity}, the rate is locally linear.}
    \label{fig:rfw}
\end{figure}

This section presents a numerical experiment illustrating the rates stated in \cref{th:RFW_geodesical_strong_convexity}. Let the manifold $\mathcal{M}$ be $\mathbb{S}^{n-1}$, the unit sphere embedded in $\mathbb{R}^n$. Let us consider the following problem.
\begin{equation}\label{eq:example_min_sphere}
    \min_{x\in\mathbb{S}^{n-1}} f(x),\qquad \text{subject to}\;\; \operatorname{dist}(x,x_c)\leq R  < \frac{\pi}{2}.
\end{equation}

Owing to the symmetries of the sphere, the linear minimization oracle for \eqref{eq:example_min_sphere} can be formulated as a simple one-dimensional problem (\cref{prop:one_dim_problem}). This problem appears, for instance, when training a neural network with spherical constraints over a hierarchical dataset.

\begin{example}[Hierarchical neural network with sphere constraints] \citet{scieur2021connecting} trained a neural network on a hierarchical dataset: \textit{``We force the classifier (hyperplane) of each class to belong to a sphere manifold, whose center is the classifier of its super-class''}. Hence, in the case wherein one wants to fine-tune the last layer of such an architecture, the problem can be formulated as in \eqref{eq:example_min_sphere}, where $x_c$ is the separating hyperplane of the super-class, $R$ is a user-defined parameter, and $f(x)$ is the loss of the neural network parametrized by $x$ over the dataset.
\end{example}

The solution of \cref{eq:example_min_sphere} is computed using the function \texttt{fminbnd} from \texttt{Matlab} (this function uses the secant method). The experimental results have been reported in \cref{fig:rfw}, wherein the two regimes (global rate of $O(t^{-1})$ and locally linearly convergent) are clearly distinguishable.

\section{Examples of Riemannian Strongly Convex Sets: Balls with Restricted Radius}\label{ex:riemannian_strongly_convex_sets} 
In this section, we always make use of the Riemannian distance $d_\mathcal{M}(x, y) = \norm{\exponinv{x}(y)}$.
The Riemannian strongly convex set, in \cref{eq:implication_def}, is the most restrictive of our definitions. It implies that, for all points $x$ in the set $\mathcal{C}$, the logarithmic image of $\mathcal{C}$ around $x$, i.e., $\exponinv{x}(\mathcal{C})$, is strongly convex in the Euclidean sense.  Intuitively, we can expect that such a situation arises when the logarithmic map $\exponinv{x}$ does not affect the shape of $\mathcal{C}$ too much, which is the case when the sectional curvature of $\mathcal{M}$ is bounded and the set $\mathcal{C}$ is not too large. 

In this section, we show that some sublevel sets of the function $x \mapsto \frac{1}{2}d_\mathcal{M}(x, x_0)^2$ are Riemannian strongly convex. Before proving the main theorem, we first introduce some concepts and two known results: \textbf{1)} distance functions are locally geodesically smooth and strongly convex \citep{martinez2022accelerated}, and \textbf{2)} locally, Euclidean pulled-back functions of geodesically smooth, strongly convex functions are smooth and strongly convex in the Euclidean sense \citep{criscitiello2022negative}.

\subsection{Bounded curvature} 
In the following, we make an assumption regarding the curvature of our manifolds. Let us recall that, given a $two$-dimensional subspace $V \subseteq T_x \mathcal{M}$ of the tangent space of a point $x$, the sectional curvature at $x$ with respect to $V$ is defined as the Gauss curvature for the surface $\expon{x}(V)$ at $x$. The Gauss curvature at a point $x$ can be defined as the product of the maximum and minimum curvatures of the curves resulting from intersecting the surface with planes that are normal to the surface at $x$. See more details on the curvature tensor $\curvtensor$ in \citep{petersen2006riemannian}. Our assumption is as follows.

\begin{assumption} \label{assump:bounded_curvature}
    The sectional curvatures of $\mathcal{M}$ are contained in the interval $[\kappa_{\min},\,\kappa_{\max}]$  and the covariant derivative of the curvature tensor is bounded as $\norm{\nabla \curvtensor} \leq F$.
\end{assumption}

This assumption is not overly restrictive. The majority of the applications of Riemannian optimization are in locally symmetric spaces, which satisfy $\nabla \curvtensor = 0$, for instance, constant curvature spaces, the SPD matrix manifold with the usual metric, $\operatorname{SO}(n)$, and the Grasmannian manifold \citep{lezcano2020curvature}.

\subsection{Strong Convexity and Smoothness of Distance Functions} We now state a fact regarding the smoothness and strong convexity of the distance squared to a point, that is central to many Riemannian optimization algorithms. In the sequel, we use the notation $K \defi \max\{\abs{\kappa_{\min}};\,\kappa_{\max}\}$).

\begin{proposition}(See \citep{martinez2022accelerated}) \label{prop:dist_func_strong_cvx_smooth}
    Let us consider a uniquely geodesic Riemannian manifold $\mathcal{M}$ of sectional curvature bounded in $[\kmin, \kmax]$ and a ball $B_{x_0}(r)$ in $\mathcal{M}$ of radius $r$ centered at $x_0$. The function $x \mapsto \frac{1}{2}d_\mathcal{M}(x, x_0)^2$ is then $\deltar{r}$-strongly convex and $\zetar{r}$-smooth in $B_{x_0}(r)$, where $\deltar{r}$ and $\zetar{r}$ are the geometric constants defined by
    \begin{equation}\label{eq:defi_zeta_and_delta}
    \zetar{r} \defi 
        \begin{cases} 
            r\sqrt{\abs{\kmin}}\coth(r\sqrt{\abs{\kmin}})& \text{ if } \kmin \leq 0 \\
            1 &  \text{ if } \kmin > 0\\
        \end{cases} ; \quad
        \deltar{r} \defi 
        \begin{cases} 
            1 & \text{ if }\kmax \leq 0\\
            r\sqrt{\kmax}\cot(r\sqrt{\kmax})& \text{ if } \kmax > 0 \\
        \end{cases}
\end{equation}
\end{proposition}
In the case of Cartan--Hadamard manifolds, $\deltar{r} = 1$ and $\zetar{r} \in [ r\sqrt{\abs{\kmin}}, r\sqrt{\abs{\kmin}} +1]$.

\subsection{Smoothness and Strong Convexity of Euclidean Pulled-Back Function} Under \cref{assump:bounded_curvature}, \cite[Proposition 6.1]{criscitiello2022negative} showed that the Euclidean pulled-back function $x \mapsto f(\expon{x_0}(x))$ of a smooth, strongly convex function $f$ defined in a ball of restricted radius in $\mathcal{M}$ is also smooth, strongly convex (in the Euclidean sense) in a ball of the same radius.

\begin{proposition} \label{prop:prop_boumal}
(Informal, see \cite[Proposition 6.1]{criscitiello2022negative} for details). Let $\mathcal{M}$ be a uniquely geodesic Riemannian manifold that satisfies \cref{assump:bounded_curvature} and that contains $\mathcal{B}(x_{\text{ref}},r)$ defined as
\[
    \mathcal{B}(x_{\text{ref}},r) : \{x\in\mathcal{M} : d_\mathcal{M}(x,x_{\text{ref}})\leq r\},\qquad \text{and let} \qquad \mathcal{B}_{x_{\text{ref}}}(0,r) : \{v\in\mathcal{T}_{x_{\text{ref}}}\mathcal{M} :\|v\|\leq r\}.
\]
As shown above, we also defined the pulled-back ball to $\mathcal{T}_{x_{\text{ref}}}\mathcal{M}$.
Let us assume the function $f:\mathcal{M} \rightarrow \mathbb{R}$ is $L$-smooth, $\mu$-strongly geodesically convex in the ball $\mathcal{B}(x_{\text{ref}},r)$, and has its minimizer $x^\star \in \mathcal{B}(x_{\text{ref}},r)$. 
Then, if 
\[
    r\leq \frac{\mu}{L} \min\left\{\frac{1}{4K}, \frac{K}{4F}\right\},
\] 
the pulled-back Euclidean function $x \mapsto f(\expon{x_0}(x))$ is $\frac{3}{2}L$-smooth, and $\frac{1}{2}\mu$-strongly convex over the ball $\mathcal{B}_{x_{\text{ref}}}(0,r)$.
\end{proposition}

One can relax the assumption of the minimizer being in the ball as this is only used to show that the Lipschitz constant of the function in the ball is at most $2Lr$. For instance, if the distance between the minimizer and $x_{\text{ref}}$ is $R$, the Lipschitz constant can be bounded by $O(LR)$, and we could use this bound to conclude a similar statement.

\subsection{Main result} Using \cref{prop:dist_func_strong_cvx_smooth,prop:prop_boumal}, we can prove that balls and other sets obtained from sublevel sets of smooth and strongly g-convex functions in small regions are Riemannian strongly convex. 

\begin{theorem}\label{thm:restricted_balls_are_riemannian_strongly_convex}
    Let $\mathcal{M}$ be a uniquely geodesic Riemannian manifold that satisfies \cref{assump:bounded_curvature}, and let the function $f: x\mapsto \frac{1}{2} d_\mathcal{M}(x,x_0)^2$, $x_0\in\mathcal{M}$. Then, the sublevel sets
    \[
        \mathcal{C}_r = \{x\in\mathcal{M} : f(x)\leq \frac{1}{2}r^2\}
    \]
    are Riemannian strongly convex if $r \leq \frac{1}{2} \frac{\deltar{r}}{\zetar{r}} \min\{\frac{1}{4K}, \frac{K}{4F}\}$, where $\zetar{r}$ and $\deltar{r}$ are the smoothness and strong convexity parameters of the function $\frac{1}{2}d_\mathcal{M}(\cdot,x_0)^2$ in $\mathcal{C}_r$ (see \cref{eq:defi_zeta_and_delta}).
\end{theorem}
\begin{proof}
    In this setting, pulling the function back to the tangent space of any point in $\mathcal{C}_r$ results in a Euclidean function that is strongly convex and smooth with condition number $O(\zetar{r}/\deltar{r})$ in $\mathcal{C}_r$. Furthermore, for all point $x \in \mathcal{C}_r$, the function $\hat{f}_x : \exponinv{x}(\mathcal{C}_r) \to \mathbb{R}$ defined as $y \mapsto f(\expon{x}(y))$ for all $y\in \exponinv{x}(\mathcal{C}_r)$ is strongly convex with the parameter $\deltar{r}/2$, and in particular, $\exponinv{x}(\mathcal{C}_r)$ is a strongly convex set. It should be noted that we considered that the distance from $x$ to any point in $\mathcal{C}_r$ is at most $2r$.
\end{proof}

This implies that Riemannian balls in these generic manifolds are geodesic strongly convex sets. We note that the greatest $r$ that satisfies the condition in \cref{thm:restricted_balls_are_riemannian_strongly_convex} is roughly a constant, considering constant curvature bounds. Being able to optimize in these sets is important as \citet{martinez2020global} proved that one can reduce global geodesically convex optimization to the optimization over these sets.

\section{Example of a Simple Linear Minimization Oracle} When the curvature is constant (spheres or hyperbolic spaces) and the domain is a ball, the linear minimization oracle can be simplified into a one-dimensional problem.

\begin{theorem}[\textbf{Linear minimization oracle in a ball in a constant curvature manifold}] \label{thm:solution_gamma}
    Let $\ball \defi \B(x_0, r)\subset \M$ be a closed Riemannian ball in a manifold $\M$ of a constant sectional curvature. If the curvature $K$ is positive, let us assume $r < \frac{\pi}{2\sqrt{K}}$, so that the ball is uniquely geodesically convex. Given a point $x$ and a direction $v \in T_x\M$, we define
    \[ 
        \Gamma \defi \expon{x}\left(\operatorname{span}\{\exponinv{x}(x_0), v\}\right) \cap \{x \ | d_\mathcal{M}(x, x_0) = r\}.
    \]
    The solution of $\argmin_{z\in\ball} \innp{v, \exponinv{x}(z)}$ can then be found in $\Gamma$.
\end{theorem}

\begin{proof}
    A solution is in the first set of the definition of $\Gamma$ owing to the symmetries of the manifold.  Indeed, let us assume there is a solution $z$, such that
    \[     
        z\notin\expon{x}\left(\operatorname{span}\{\exponinv{x}(x_0), v\}\right).
    \]
    Now, let us consider the points $\exponinv{x}(z)\in T_x \M$, and $\exponinv{x}(z')\in T_x \M$, such that $z'$ is the point resulting from the application of $\expon{x}(\cdot)$ to the symmetric of $\exponinv{x}(z)$ with respect to the plane  $\operatorname{span}\{\exponinv{x}(x_0), v\}$. 
    
    Then, $z'$ is also a solution; therefore, $\expon{x}(\frac{1}{2}\exponinv{x}(z) + \frac{1}{2}\exponinv{x}(z'))$ would also be a solution, and it would be in $\expon{x}\left(\operatorname{span}\{\exponinv{x}(x_0), v\}\right)$. 
    
    Moreover, any solution $z$ satisfies $d_\mathcal{M}(x, z) = r$ because, otherwise, there exists a neighborhood around $z$ contained in $\B$, and we could further decrease the function value of $\innp{v, \exponinv{x}(\cdot)}$ for some point in this neighborhood. Finally, we can parametrize $\Gamma$ as $\theta \mapsto $ the point of intersection of $\B$ with the geodesic segment starting at $x$ from the direction $\cos(\theta)\exponinv{x}(x_0)+ \sin(\theta) v$, where $\theta \in [0, 2\pi)$. Except for the degenerate case wherein $\exponinv{x}(x_0)$ is parallel to $v$, where $\Gamma$ is a single point and the solution we are looking for.
\end{proof}

A direct consequence of \cref{thm:solution_gamma} is that the problem of approximating a solution of the linear max oracle can be solved by solving an alternative one-dimensional problem.

\begin{proposition}\label{prop:one_dim_problem}
    Let $u_1,\,u_2$ be an orthonormal basis for the linear subspace $\operatorname{span}\{\exponinv{x}(x_0), v\}$.
    
    The solution of $\argmin_{z\in\ball} \innp{v, \exponinv{x}(z)}$ can then be obtained by solving the following one-dimensional problem,
    \[
        \min_{\phi\in[-\pi,\pi]}\alpha(\phi)\langle v; p(\phi)\rangle,
    \]
    where $p(\phi)$ and $\alpha(\phi)$ are defined as
    \begin{align}
        p(\phi) & = \cos(\phi) u_1 + \sin(\phi)u_2,\qquad \phi\in[-\pi,\pi],\\
        \alpha(\phi) & = \argmin_{\alpha>0}: \operatorname{Dist}\left(\operatorname{Exp}_x(\alpha p(\phi)),x_0\right)=r. \label{eq:alpha_phi}
    \end{align}
\end{proposition}
\begin{proof}
    The proof is immediate. The subspace $\operatorname{span}\{\exponinv{x}(x_0), v\}$ is parametrized with polar coordinates: $p$ is a direction of the unitary norm, and $\alpha$ is the length. The non-linear equation \eqref{eq:alpha_phi} represents the smallest $\alpha$ such that $\alpha\cdot p$ intersects $\mathcal{B}(x_0,r)$.
\end{proof}

The solution $\phi$ can be computed using a one-dimensional solver, e.g., bisection or the Newton-Raphson method. It should be noted that finding $\alpha(\phi)$ is also a one-dimensional problem, and it can also be solved using similar techniques. 

In some cases, there are simple formulas for $\alpha(\phi)$, e.g., when the manifold is a sphere (\cref{prop:def_alpha}).

\begin{proposition} \label{prop:def_alpha}
    Let $\mathcal{M}=\mathbb{S}^{n-1}$ be the unit sphere. The formula of \cref{eq:alpha_phi} then reads as
    \[
        \alpha(\phi) = 2\text{tan}^{-1}\left( \frac{b(\phi)+\sqrt{a^2+b(\phi)^2-c^2}}{a+c} \right),\quad a = x_0^Tx,\quad b(\phi)=x_0^Tp(\phi),\quad c = 1-2\sin\left( \frac{r}{2} \right).
    \]
\end{proposition}
\begin{proof}
    The distance function and exponential map in the sphere read as
    \[
        \operatorname{Dist}(x,y) = 2\,\text{asin}\left(\frac{\|y-x\|}{2}\right),\qquad \operatorname{Exp}_x(tv) = x \cos(t\|v\|) + \frac{tv}{\|tv\|}\sin(t\|v\|),\qquad v\in\TxM,\;\; t\in\mathbb{R}.
    \]
    Hence, for all $x,\,y\in\mathcal{M}$,
    \begin{align*}
        & \operatorname{Dist}(x,y)=2\,\text{asin}\left(\frac{\|y-x\|}{2}\right) = r,\\
        \Rightarrow & \|y-x\| = 2\sin\left( \frac{r}{2} \right),\\
        \Rightarrow & \|y\|^2+\|x\|^2-2x^Ty = 4\sin^2\left( \frac{r}{2} \right),\\
        \Rightarrow & x^Ty = 1-2\sin^2\left( \frac{r}{2} \right).
    \end{align*}
   By replacing $x,\,y$ by $\operatorname{Exp}_x(\alpha p),\, x_0$ and using the expression of the exponential map, we obtain the equation
    \begin{align*}
        x_0^Tx\cos(\alpha)+x_0^Tp(\phi)\sin(\alpha) = 4\sin^2\left( \frac{r}{2} \right).
    \end{align*}
    The desired result follows after identifying the terms $a,\,b,\,c$ in
    \[
        a\cos(\alpha) + b\sin(\alpha) = c,
    \]
    the solution of which is 
    \[
        \alpha = 2\text{tan}^{-1}\left( \frac{b\pm\sqrt{a^2+b^2-c^2}}{a+c} \right)+2k\pi,\qquad k\in\mathbb{Z}.
    \]
\end{proof}

\section{Conclusion}\label{sec:conclusion}
We presented the first definitions for strong convexity of sets in Riemannian manifolds, studied their relationships, and have provided examples of these sets.
Our definitions seek to be well-suited for optimization and for establishing the strongly convex nature of the set.
The global linear convergence of the RFW algorithm serves as a tangible demonstration of the impact of developing a theory around the strong convexity structure of these sets.

However, most importantly, we expect the development of a strongly convex structure to be helpful when developing Riemannian algorithms in the contexts wherein the Euclidean algorithm counterpart was leveraging such a structure, e.g., in the generalized power method \citep{journee2010generalized}, online learning \citep{huang2017following,dekel2017online,bhaskara2020online,bhaskara2020onlineMany}, and even more broadly, in the case of the use of \textit{strongification} techniques, as in \citep{Molinaro20}.

\subsection*{Acknowledgments}

David Martínez-Rubio was partially funded by the DFG Cluster of Excellence MATH+ (EXC-2046/1, project id 390685689) funded by the Deutsche Forschungsgemeinschaft (DFG). AdA would like to acknowledge support from a Google focused award, as well as funding by the French government under management of Agence Nationale de la Recherche as part of the "Investissements d'avenir" program, reference ANR-19-P3IA-0001 (PRAIRIE 3IA Institute). We would like to thank Editage (www.editage.co.kr) for English language editing.

\clearpage

{\small
\ifdefined\isArxiv
\bibliographystyle{abbrvnat}
\fi
\bibsep 1ex
\bibliography{utils/bibliography}}

\clearpage

\appendix

\section{Technical Proposition and Lemmas}

\subsection{Double Geodesic Strong Convexity implies Riemannian Scaling Inequality}

\begin{restatable}[Double Geodesic Str. Cvx. implies Riemannian Scaling Inequality]{proposition}{doublegeodesicimplyscaling}\label{prop:implication_riemannian_scaling_inequality}
Let $\mathcal{C}\subset\mathcal{M}$ be a double geodesic $\alpha$-strongly convex set (\cref{def:double_geodesic_strong_convexity}) in a complete connected Riemannian manifold $\mathcal{M}$. Let us assume that the double exponential map operator (\cref{def:double_exponential_map}) is such that
\begin{equation}\label{eq:condition_exponential_operator_map}
h_x(u,v) = u + v,~~\forall(u,v)\in \TxM .
\end{equation}
The Riemannian scaling inequality in \cref{def:riemannian_scaling_inequality} is then satisfied.
\end{restatable}

\begin{proof}
Let us consider $x\in\mathcal{C}$, $w\in \TxM $, and $v\in\mathcal{C}$ s.t.
\[
v \in \underset{z\in\mathcal{C}}{\operatorname{argmax}} \langle w; \operatorname{Exp}^{-1}_x(z)\rangle_x.
\]
Then, by \eqref{eq:geodesic_strong_convexity_double_exponential}, which is equivalent to \cref{def:double_geodesic_strong_convexity} with $t=1/2$,  we have $\operatorname{Exp}_{\gamma_{x,v}(1/2)}\Big(\frac{\alpha}{4}d^2(x, v) z\Big)\in\mathcal{C}$, where $\gamma_{x,v}(\cdot)$ is the geodesic joining $x$, and $v$ and $z$ are a unit norm vector in $T_{\gamma_{x,v}(1/2)}\mathcal{M}$.
Then, by optimality of $v$, for all $z\in T_{\gamma_{x,v}(1/2)}\mathcal{M}$ with $\|z\|_{\gamma_{x,v}(1/2)}=1$, we have
\[
\Big\langle w; \operatorname{Exp}^{-1}_x(v)\Big\rangle_x \geq \Big\langle w; \operatorname{Exp}^{-1}_x\Big(\operatorname{Exp}_{\gamma_{x,v}(1/2)}\big(\frac{\alpha}{4}d^2(x, v) z\big)\Big)\Big\rangle_x.
\]
Let us first recall that, by definition of the exponential map, for the geodesic $\gamma_{x,v}$ and for all $t\in[0,1]$, we have $\operatorname{Exp}_x(t\gamma^\prime_{x,v}(0)) = \gamma_{x,v}(t)$ such that we can write 
\begin{equation}\label{eq:optimality_inequality}
\Big\langle w; \operatorname{Exp}^{-1}_x(v)\Big\rangle_x \geq \Big\langle w; \operatorname{Exp}^{-1}_x\Big(\operatorname{Exp}_{\operatorname{Exp}_x(\frac{1}{2}\gamma^\prime_{x,v}(0))}\big(\frac{\alpha}{4}d^2(x, v) z\big)\Big)\Big\rangle_x.
\end{equation}
We can hence write this in terms of the double exponential map and subsequently in terms of the exponential operator map. We use 
\[
\operatorname{Exp}_{\operatorname{Exp}_x(\frac{1}{2}\gamma^\prime_{x,v}(0))}\big(\frac{\alpha}{4}d^2(x, v) z\big) = \operatorname{Exp}_x\Big(\frac{1}{2}\gamma^\prime_{x,v}(0); \big(\Gamma_x^{\gamma_{x,v}(1/2)}\big)^{-1}(\frac{\alpha}{4}d^2(x,v)z) \Big).
\]
Hence, on using the assumption \eqref{eq:condition_exponential_operator_map} on the exponential operator, we have
\[
\operatorname{Exp}_{\operatorname{Exp}_x(\frac{1}{2}\gamma^\prime_{x,v}(0))}\big(\frac{\alpha}{4}d^2(x, v) z\big) = \operatorname{Exp}_x\Big(\frac{1}{2}\gamma^\prime_{x,v}(0) + \big(\Gamma_x^{\gamma_{x,v}(1/2)}\big)^{-1}(\frac{\alpha}{4}d^2(x,v)z) \Big).
\]
When plugging the last equality in \eqref{eq:optimality_inequality}, we obtain
\[
\Big\langle w; \operatorname{Exp}^{-1}_x(v)\Big\rangle_x \geq \frac{1}{2} \Big\langle w; \gamma^\prime_{x,v}(0)\Big\rangle_x + \Big\langle w; \big(\Gamma_x^{\gamma_{x,v}(1/2)}\big)^{-1}(\frac{\alpha}{4}d^2(x,v)z) \Big\rangle_x.
\]
It should be noted that $\gamma^\prime_{x,v}(0)= \operatorname{Exp}^{-1}_x(v)$, such that
\[
\Big\langle w; \operatorname{Exp}^{-1}_x(v)\Big\rangle_x \geq 2\Big\langle w; \big(\Gamma_x^{\gamma_{x,v}(1/2)}\big)^{-1}(\frac{\alpha}{4}d^2(x,v)z) \Big\rangle_x.
\]
Wth $\Big(\Gamma_x^y\Big)^{-1} = \Gamma_y^x$ and the isometry property of $\Gamma_x^{\gamma_{x,v}(1/2)}$, we have
\begin{eqnarray}
\Big\langle w; \Gamma_{\gamma_{x,v}(1/2)}^x \big(\frac{\alpha}{4}d^2(x,v)z\big) \Big\rangle_x & = & \Big\langle \Gamma_{\gamma_{x,v}(1/2)}^{x}\Gamma_{x}^{\gamma_{x,v}(1/2)} w ; \Gamma_{\gamma_{x,v}(1/2)}^x \big(\frac{\alpha}{4}d^2(x,v)z\big) \Big\rangle_x\\
& =& \Big\langle\Gamma_{x}^{\gamma_{x,v}(1/2)} w; \frac{\alpha}{4}d^2(x,v)z \Big\rangle_{\gamma_{x,v}(1/2)}.
\end{eqnarray}
Hence, for all $z$ of a unit norm in $T_{\gamma_{x,v}(1/2)}\mathcal{M}$, we obtain
\begin{eqnarray}
\Big\langle w; \operatorname{Exp}^{-1}_x(v)\Big\rangle_x &\geq& 2\Big\langle \Gamma_x^{\gamma_{x,v}(1/2)} w; \frac{\alpha}{4}d^2(x,v)z \Big\rangle_{\gamma_{x,v(1/2)}} \\
&=& \frac{\alpha}{2} d^2(x,v)  \Big\langle \Gamma_x^{\gamma_{x,v}(1/2)} w; z \Big\rangle_{\gamma_{x,v(1/2)}}.
\end{eqnarray}
Furthermore, by maximizing over $z$, for the best $z^*$, we obtain
\[
\Big\langle \Gamma_x^{\gamma_{x,v}(1/2)} w; z^* \Big\rangle_{\gamma_{x,v(1/2)}} = \|\Gamma_x^{\gamma_{x,v}(1/2)} w\|_{\gamma_{x,v}(1/2)}.
\] 
Then, because the parallel transport $\Gamma_x^{\gamma_{x,v}(1/2)}$ is an isometry, we finally have
\[
\Big\langle w; \operatorname{Exp}^{-1}_x(v)\Big\rangle_x \geq \frac{\alpha}{2} d^2(x,v) \|\Gamma_x^{\gamma_{x,v}(1/2)} \nabla f(x)\|_{\gamma_{x,v}(1/2)} = \frac{\alpha}{2} d^2(x,v) \|w\|_x.
\]
\end{proof}

\subsection{Proof of Smoothness Property Lemma} \label{ap:proofs_level_set}

First, we introduce the Fenchel conjugate of a function defined on a manifold.

\begin{definition} \citep{bergmann2021fenchel}
   Let us Suppose that $f:\mathcal{C}\rightarrow \mathbb{R}$, where $\mathcal{C}\subset \mathcal{M}$ is a strictly convex set, where $\mathcal{M}$ is a Cartan--Hadamard manifold. For $m\in\mathcal{M}$, the $m$-Fenchel conjugate of $f$ is defined as the function $f_m^*:T_m^*\mathcal{M}\rightarrow \mathbb{R}$ such that
    \begin{equation}\label{eq:Riemannian_conjugate}
        f_m^*(\xi_m) := \sup_{x\in T_m\mathcal{M}}\left\{ \langle \xi_m,\, x \rangle - f\left( \operatorname{Exp}_m x \right) \right\},
    \end{equation}
    where $T_m^*\mathcal{M}$ is the cotangent bundle of $T_m\mathcal{M}$
\end{definition}
In particular, we need the following property.
\begin{lemma}\cite[lem. 3.7]{bergmann2021fenchel} \label{lem:fenchel_dual_inequality}
   Let us suppose that $f,\tilde f:\mathcal{C}\rightarrow \mathbb{R}$ are proper functions, where $\mathcal{C}\subset \mathcal{M}$ is a strictly convex set, and $\mathcal{M}$ is a Cartan--Hadamard manifold, and let $m\in\mathcal{C}$. Then,
    \begin{equation}
        \text{if } f(p) \leq \tilde f(p) \;\;\forall \;p\in\mathcal{C}, \text{ then } f^*_m(\xi_m) \geq \tilde f^*_m(\xi_m) \;\; \forall \; \xi_m \in T_m^*\mathcal{M}.
    \end{equation}
\end{lemma}

We are now ready to prove \cref{lem:L_smoothness_bound}.

\Lsmoothnessbound*
\begin{proof}
As $f$ is a geodesically smooth function, for all $p\in\mathcal{C}$, we have $f(p)\leq \tilde f(p)$, where
\[
    \tilde f(p) = f(x) + \langle \nabla f(x),\, \operatorname{Exp}_x^{-1}(p)\rangle + \frac{L}{2} d^2(x,p), \quad \forall p\in\mathcal{C}.
\]
Therefore, according to \cref{lem:fenchel_dual_inequality}, we have $f_m^*(\xi_m) \geq \tilde f_m^*(\xi_m)$, for all $m\in\mathcal{M}$. Using the definition \eqref{eq:Riemannian_conjugate},
\begin{align*}
    \tilde f_m^*(\xi_m) & = \sup_{Z\in T_m\mathcal{M}}  \langle \xi_m,\, Z \rangle_x - \tilde f(\operatorname{Exp}_m(Z))
\end{align*}
In the particular case wherein $m=x$, we obtain
\begin{align*}
    \tilde f_x^*(\xi_x) & = \sup_{Z\in T_m\mathcal{M}} \langle \xi_x,\, Z \rangle - f(x) - \langle \nabla f(x),\, Z\rangle_x - \frac{L}{2} \|Z\|_x^2 \\
    & = - f(x) + \frac{1}{2L}\|\xi_x-\nabla f(x)\|_x^2
\end{align*}
Therefore, the inequality $f^*_x(\xi_x) \geq  \tilde f_x^*(\xi_x)$ can be written as
\[
     \frac{1}{2L}\|\xi_x-\nabla f(x)\|_x^2 \leq f^*_x(\xi_x) + f(x).
\]
In particular, at $\xi_x = 0$, we have the desired result, as
\[
    f^*_x(0) = \sup_{Z\in T_m\mathcal{M}}  \langle 0,\, Z \rangle_x - f(\operatorname{Exp}_x(Z)) = \sup_{Z\in T_m\mathcal{M}}  - f(\operatorname{Exp}_x(Z)) = -\inf_{z\in \mathcal{C}}  f(z) = -f^*.
\]
\end{proof}

\clearpage
\section{MATLAB code for the experiments}

\subsection{Subroutine: Linear Max Oracle For a Ball In a Sphere}

\begin{lstlisting}
function v = linear_max_oracle(w, x, R, x0, manifold)

% Solve the problem
% max w^Tz : z \in exp_x^{-1}(C),
% where C := {x:dist(x,x_0)\leq R}
% Assume that the sphere has radius = 1!

% Define the two basis vector of the subspace u_1 and u_2
u1 = w;
u1 = u1/norm(u1);
u2 = manifold.log(x, x0);
u2 = u2-u1*(u1'*u2);

if norm(u2) > 0
    u2 = u2/norm(u2);
    p = @(phi) cos(phi)*u1 + sin(phi)*u2;

    alpha = @(p) solve_sincoseq(x0'*x,x0'*p,(1-2*sin(R/2)^2));

    fhandle = @(phi) -(alpha(p(phi))*cos(phi)); %same minimum since w*u2 = 0.
    options = optimset('TolX',eps);
    [best_phi] = fminbnd(fhandle, -pi, pi, options);
    v = manifold.exp(x, p(best_phi), alpha(p(best_phi)));
else
    alpha = solve_sincoseq(x0'*x,x0'*u1,(1-2*sin(R/2)^2)); % best direction: u1
    v = manifold.exp(x, u1, alpha);
end
\end{lstlisting}

\clearpage

\subsection{Sample Script: Random Quadratic Over a Sphere With Ball Constraint}
~\nocite{boumal2014manopt}
\begin{lstlisting}
% Generate the problem data.
d = 50;
n = 25; % Slower convergence when n < d -> non-strongly convex case
nIter = 500;
dualgap =  zeros(1, nIter);

% Manifold: this code uses manop, see https://www.manopt.org
manifold = spherefactory(d); 

% Define the problem cost function and its derivatives.
A = randn(n,d);
A = A'*A;
A = A/norm(A);
xstar = manifold.rand();
f = @(x) 0.5*(x-xstar)'*A*(x-xstar);
mgrad = @(x) manifold.egrad2rgrad(x, A*(x-xstar));
L = norm(A); % Worst case

% Create the problem set
x_center = manifold.rand();
radius_ratio = 0.9; % <1: xstar is oustide
radius_max = manifold.dist(x_center,xstar)*radius_ratio;
setFunction = @(x) manifold.dist(x_center, x)<=radius_max;

x = x_center;
% Main loop RFW
for i=1:(nIter-1)
    gradx = mgrad(x);
    v = linear_max_oracle(-gradx, x, radius_max, x_center, manifold);
    dualgap(i) = -manifold.inner(x, gradx, manifold.log(x, v));
    
    step_size = -manifold.inner(x, gradx, manifold.log(x, v)) / (L*manifold.dist(x, v)^2);
    step_size = min(step_size, 1);
    x = manifold.exp(x, manifold.log(x, v), step_size);
end
v = linear_max_oracle(-gradx, x, radius_max, x_center, manifold);
dualgap(end) = -manifold.inner(x, gradx, manifold.log(x, v));

% Max at eps, otherwise the result is numerically meaningless
semilogy(1:length(dualgap), max(dualgap,eps))
legend({'FW Dual Gap'})
\end{lstlisting}

\end{document}